\documentclass[a4paper]{article}

\usepackage[utf8]{inputenc}
\usepackage{lmodern}
\usepackage[colorlinks]{hyperref}
\usepackage[english]{babel}
\usepackage[margin=26mm]{geometry}

\usepackage{tikz}

\usepackage[colorlinks]{hyperref}
\usepackage{enumerate}
\usepackage{graphicx}
\usepackage{amssymb, amsfonts, amsthm, amsmath}
\theoremstyle{plain}
\newtheorem{theorem}{Theorem}
\newtheorem{proposition}[theorem]{Proposition}
\newtheorem{lemma}[theorem]{Lemma}
\newtheorem{corollary}[theorem]{Corollary}

\theoremstyle{definition}

\theoremstyle{remark}
\newtheorem{remark}[theorem]{Remark}

\title{A generalized model interpolating  between the random energy model and the branching random walk}
\author{Mohamed Ali Belloum\thanks{\texttt{belloum@math.univ-paris13.fr}}}
\date{ }

\begin{document}
	\maketitle
	\begin{abstract}
		We study a generalization of the model introduced in \cite{zbMATH06473046} that interpolates between the random energy model (REM) and the branching random walk (BRW). More precisely, we  are interested in the asymptotic behaviour of the extremal process associated to this model. In \cite{zbMATH06473046}, Kistler and Schmidt show that the extremal process of the $GREM(N^{\alpha})$, $\alpha\in{[0,1)}$ converges weakly to a simple Poisson point process. This contrasts with the extremal process of the branching random walk $(\alpha=1)$ which was shown to converge toward a \textit{decorated} Poisson point process by Madaule \cite{zbMATH06705452}. In this paper we propose a generalized model of the $GREM(N^{\alpha})$, that has the structure of a tree with $k_n$ levels, where $(k_n\leq n)$ is a non-decreasing sequence of positive integers. We study a generalized case, where the position of the particles are not necessarily Gaussian variables and the reproduction law is not necessarily binary. We show that as long as $b_n=\lfloor{\frac{n}{k_n}}\rfloor\to_{n\to \infty}\infty$  
		  in the Gaussian case (with the assumption  $\frac{b_n}{\log(n)^2}\to\infty$ as $n\to \infty$ in the non Gaussian case) the decoration disappears and we have convergence to a simple Poisson point process. 
		\vspace{0.5cm}
		

	\end{abstract}
	 	\section{Introduction}
 	The random energy model (REM) was introduced by Derrida in $1981$ \cite{zbMATH06494911} for the study of spin glasses. In the REM, there are $2^N$ spin configurations. Each configuration $\sigma \in{\left\{-1,1\right\}}^{N}$corresponds  to an independent centred  Gaussian random variable $X_{\sigma}$ with variance $N$, that models its energy level. It is well-known  that the extremal process of the REM, which is defined as  \begin{equation}
  \label{eqn:cvgRem}
  \mathcal{E}_N = \sum_{\sigma \in \{-1,1\}^N} \delta_{X_\sigma - m_N},
\quad \text{where } m_N = \beta_c N - \frac{1}{2\beta_c} \log(N) \text{
and } \beta_c = \sqrt{2\log(2)},
\end{equation}
	converges weakly in distribution to a Poisson point process with intensity $\frac{1}{\sqrt{2\pi}} e^{-\beta_c x}dx$. Additionally the law of the  maximum $M_N=\max_{\sigma \in{\left\{-1,1\right\}}^{N}}X_{\sigma}$ centred by $m_{N}$ converges weakly to a Gumbel random variable. 
	
	Derrida introduced a generalized model in $1985$, called the GREM \cite{derrida1985generalization}, that has the structure of a tree with $K$ levels ($K$ is a fixed constant in $\mathbb{N}^*$) and can be described as follows. Start by a unique individual (the root). It gives birth to $2^{\frac{N}{K}}$ (we assume that  $\frac{N}{K}$ is a positive integer) children at the first level.  At each level $i$, $1\leq i<K$, each child gives birth independently to $2^{\frac{N}{K}}$ children. We associate  each branch of this tree to an independent centred Gaussian random variable with variance $\frac{N}{K}$. 
	In the context of spin glasses, we obtain $2^N$ configurations in the level $K$, and the 
	level energy of each configuration is the sum of the values along the branches that forms the path from the root of the tree to the leaf corresponding to this configuration. We call this model $GREM_N(K)$. The REM in this case can be thought of as a 
	GREM with one level, i.e. a $GREM_N(1)$. The 
		correlation of the energy of two different configurations depends on the number of common branches shared by their paths from the root up to the node at which they split. These correlations   
	do not have any impact on the extreme values of the energy
	levels, as the result described in \eqref{eqn:cvgRem} still holds even
if $(X_\sigma, \sigma \in \{-1,1\}^N)$ is distributed as a $GREM_N(K)$,
as $N \to \infty$.

	 Kistler and Schmidt \cite{zbMATH06473046} studied the asymptotic of the extremal process of a GREM with a number of levels $K_N=N^{\alpha}$, for $\alpha \in{[0,1)}$. They proved that, setting
	$$m^{(\alpha)}_N=\beta_cN-\frac{2\alpha+1}{2\beta_c}\log(N),$$  the extremal process of the $GREM_N(N^{\alpha})$  converges weakly to a Poisson point process with intensity $\frac{1}{\sqrt{2\pi}}e^{-\beta_cx}dx$, and the law of the maximum converges to a Gumbel distribution. In the $GREM_N(N^{\alpha})$ the stronger
	correlations between the leaves of the tree have the effect of
	decreasing the median of the maximal energy level, specifically its
	logarithmic correction. However the limiting law of the extremal process
	remains unchanged.  
	  In the case of $\alpha=1$, which corresponds to the classical binary branching random walk, the asymptotic behaviour of the extremal process is well-known. The convergence in law of the recentred maximum was proved by Aidékon \cite{zbMATH06216112}, and recently Madaule \cite{zbMATH06705452} showed the convergence of the extremal process to a decorated Poisson point process with random intensity. Therefore a phase transition can be
	  exhibited, from a simple Poisson point process appearing in the
	  $GREM_N(N^\alpha)$ for $\alpha < 1$ to a decorated one for $\alpha = 1$.
	
	The aim of this article is to have a
	closer look at this phase transition.
	We take interest in a
	generalized version  
	of the $GREM_N(N^{\alpha})$, that has the structure of a tree with $k_n$ levels, where $(k_n\geq 0)$ is a non-decreasing sequence of positive integers. We study the asymptotic behaviour of the extremal point process showing that as long as $\frac{k_n}{n}\to_{n\to\infty}0$ (in the Gaussian case), the decoration does not appear.
	\section{Notation and main result}
	A branching random walk on $\mathbb{R}$ is a particle system that evolves as follows. It starts with a unique
	individual located at the origin at time 0. At each time $n\geq 1$, each individual alive in the process dies and gives birth to a random number of children, that are positioned around their parent according to i.i.d random variables.
	
	The process we take interest in  can be described as follows. Let $k_n$ be an integer sequence growing to $\infty$ such that $k_n\leq n$ for all $n\in{\mathbb{N}}$ and set $b_n=\lfloor{\frac{n}{k_n}}\rfloor$ the integer part of $\frac{n}{k_n}$. The process starts with a unique individual located at the origin at time 0.  
	The particles reproduce for $b_n$ consecutive steps, each particle giving birth to an i.i.d. number of children. Then each descendant of the initial ancestors moves independently, making $b_n$ i.i.d. steps of displacements. This forms the first generation of the process. For each $1 \leq
	k \leq k_n$, every individual at generation $k$ repeats independently of
	the others the same reproduction and displacement procedure as the
	original ancestor. In other words every individual creates a number of
	descendants given by the value at time $b_n$ of a Galton-Watson process,
	whose positions are given by i.i.d. random variables with the same law
	as a random walk of length $b_n$.
	
	 To describe the model formally we introduce Ulam-Harris notation for trees.
	  Set $$ \mathcal {U} = \bigcup_ {n \geq0} \mathbb {N} ^ {n} $$ with $\mathbb{N}^{0} = \{\varnothing \} $ by convention. The element $ (u_1, u_2 .., u_n) $ represents the $ u_n ^ {\mathrm{th}} $ child of $ u_ {n-1} ^ {\mathrm{th}} $ child .., of $ u_1 ^ {} $ of the root particle which is  noted $ \varnothing $. If $u=(u_1, u_2 .., u_n)$ we denote by $u_{k}=(u_1, u_2 .., u_k)$ the sequence consisting of the $k^{\mathrm{th}}$ first values of $u$ and by $|u|$ the generation of $u$. For $u,v$ $\in{\mathcal{U}}$ we denote by $\pi(u)$ the parent of $u$. If $u=(u_1, u_2 .., u_n)$ and $v=(v_1, v_2 .., v_n)$, then we write $u.v=(u_1, u_2 .., u_n,v_1, v_2 .., v_n)$ for the concatenation of $u$ and $v$. We write $$|u\land v|:=\inf\{j\leq n: u_{j}=v_{j} \text{ and } u_{j+1}\ne v_{j+1} \}.$$ This quantity  is called the overlap of $u$ and $v$ in the context of spin glasses.
	A tree $\mathcal{T}$ is a subset of $\mathcal{U}$ satisfying the following assumptions:
	\begin{itemize}
		\item $\varnothing$ $\in{\mathcal{T}}$.
		\item if $u\in{\mathcal{T}}$, then $\pi(u) \in{\mathcal{T}}$.
		\item if $u=(u_1,u_2,...u_n)\in{\mathcal{T}}$, then $\forall$ $j \leq u_n$, $\pi(u).j\in{\mathcal{T}}$.
	\end{itemize}
	
	We  now introduce the reproduction and displacement laws associated to our process. Let $(Y_n)_{n\in{\mathbb{N}}}$ be a random walk such that $\mathbb{E}(Y_1)=0$ and $\mathrm{Var}(Y_1)=1$. We denote by $(Z_n)_{n\in{\mathbb{N}}}$ a Galton-Watson process such that $Z_0=1$  and offspring law given by the weights $(p(k))_{k\in{\mathbb{N}}}$ with $p_{0}=0$. Under this assumption, the Galton Watson process survives almost surely. Set $m=\sum_{k\geq 1}kp(k)$ the mean of the offspring distribution and assume that $m>1$. Recall that the Galton-Watson process $(Z_n)_{\in{\mathbb{N}}}$ 	satisfies for all $n\in{\mathbb{N}}$:
	$$Z_{n+1}=\sum_{j=1}^{Z_n} \xi_{n+1,j},$$ 
	where $(\xi_{n,j})_{1\leq j\leq Z_n }$ are i.i.d random variables with law $(p(k))_{k\in{\mathbb{N}}}$  .
	
	Under the assumption $\mathbb{E}(Z_1\log(Z_1))<\infty$,  Kesten and Stigum \cite{zbMATH03322553} proved that on the set of non extinction of $\mathcal{T}$ there exists a positive random variable $Z_{\infty}$ such that \begin{align}
	\label{500}
	  &\lim_{b\rightarrow \infty }\frac{Z_{b}}{m^{b}}=Z_{\infty}>0, \hspace{0.25cm} \text{a.s}.
	  \end{align}
	In this article we assume that the following stronger condition holds:
	\begin{align}
	\label{F}
 &\mathbb{E}(Z_1^2)<\infty,
	\end{align}
	 which is needed for the proof of Lemma \ref{11}.
	 
	Construct a tree that we denote $\mathcal{T}^{(n)}$ as follows.  Start by the ancestor $\varnothing$ located at the origin. It gives birth to $Z_{b_n}$ children.  For each $k\leq k_n$, each individual at the generation $k$ gives birth to an independent copy of  $Z_{b_n}$, that are positioned according to i.i.d random variables with the same law as $Y_{b_n}$.     
	For $1\leq k\leq k_n$, let $$\mathcal{H}_{k}:=\{u\in{\mathcal{T}^{(n)}}: |u|=k\},$$ the set of particles in the $k^{th}$ generation. By construction, we have $\#\mathcal{H}_{k}=Z_{kb_n}$ in law for all $k\leq k_n$.
	We define $(X^{(n)}_u, u \in \mathcal{T}^{(n)})$ a family of i.i.d.
	random variables with same law as $Y_{b_n}$. For $u \in
	\mathcal{T}^{(n)}$, we write
	\[
	S_u^{(n)} = \sum_{k=1}^{|u|} X^{(n)}_{u_k}.\]
	 The goal of this paper is to study  the asymptotic behaviour of the extremal process associated to this model $$\mathcal{E}^{(b_n)}_n=\sum_{u\in{\mathcal{H}_{k_n}}}\delta_{S^{(n)}_u-m_n},$$
	 where  the value of the median $m_n$ is given in Theorem $\ref{Theorem}$.

	 Let us introduce notation associated to the displacement of the process.  
	 For all $\theta>0$ we set \begin{align}
\label{H} \Lambda(\theta):=\log\left(\mathbb{E}\left(\exp(\theta Y_1)\right)\right).
\end{align}
We assume that there exists $\theta>0$ such that $\Lambda(\theta)<\infty$. We write:
 $$\kappa_n{(\theta)}=\log\mathbb{E}\left(\sum_{|u|=1}e^{\theta X^{(n)}_{u}}\right).$$ Observe that $\kappa_n(\theta)=b_n(\log(m)+\Lambda(\theta))$ as $$\mathbb{E}\left(\sum_{|u|=1}e^{\theta X^{(n)}_{u}}\right)=\mathbb{E}\left(\sum_{|u|=1}\mathbb{E}(e^{\theta X^{(n)}_{u}}|Z_{b_n})\right)=\mathbb{E}\left(Z_{b_n}\mathbb{E}(e^{\theta Y_{b_n}})\right)=e^{b_n(\log(m)+\Lambda(\theta))}.$$
 The function $\kappa_n$ is convex and differentiable on $\left\{\theta>0, \kappa_n(\theta)<\infty\right\}$, its interval of definition.
 We assume that there exists $\theta^{*}>0$ such that
\begin{align}
\label{K} &\theta^{*}\Lambda'{(\theta^{*})}-\Lambda(\theta^{*})=\log(m).
\end{align}
 We also assume that there  exists $\delta>0$ such that
\begin{align}
\label{equation 6}
&\mathbb{E}\left(\exp((\theta^{*}+\delta) Y_1)\right)< \infty
\end{align}
 Recall that the case $k_n=n$ corresponds to the classical branching random walk.
 Then under assumption \eqref{H} and \eqref{K}, Kingman \cite{zbMATH03509558}, Hammersley \cite{zbMATH03474678} and Biggins \cite{10.2307/1426138} showed that on the set of non-extinction of $\mathcal{T} $$$\lim_{n\to\infty}\frac{M_n}{n}:=\frac{\kappa{(\theta^*)}}{\theta^*}=v \hspace{0.25 cm} \text{a.s,}$$ where, $M_n=\max_{u\in{\mathcal{H}_{n}}}S^{}_u$ and $v$ is the speed of the right-most individual. Then, Hu and Shi  \cite{zbMATH05558299} and Addario-Berry and Reed \cite{zbMATH05587823} proved that 
 $$M_n=nv-\frac{3}{2\theta^{*}}\ln(n)+O_{\mathbb{P}}(1),$$ where $O_{\mathbb{P}}(1)$ represents a tight sequence of random variables.  	

	Throughout this paper we will assume that we are in one of the two cases: 
	
	$(\mathbf{H_1})$: $Y_1$ is a standard Gaussian variable and $b_n\to \infty$ as $n\to \infty$.
	
	$(\mathbf{H_2})$: The characteristic function $\phi(\lambda)=\mathbb{E}\left(\exp(i\lambda Y_1)\right)$ of $Y_1$ satisfies the Cramér condition, i.e $$\limsup_{|\lambda|\to \infty}|\phi(\lambda)|<1,$$  
	and $\frac{b_n}{\log(n)^2}\to\infty$ as $n\to \infty.$ 
	The last assumption on $Y_1$ (under $\mathbf{(H_2)}$) comes from the fact that we used a refined version of the Stone's local limit theorem introduced in \cite[Theorem 2.1]{borovkov2017generalization}, more precisely in Corollary \ref{stonebis}.
	
 Our work is inspired by the recent works on the convergence of the extremal processes \cite{zbMATH06083948}, \cite{zbMATH06247828}, \cite{zbMATH06473046} and \cite{zbMATH06705452}. 
 The main result of this paper is the following convergence in distribution.
	\begin{theorem}
		\label{Theorem}
		Assume that \eqref{F}, \eqref{H}, \eqref{K}, \eqref{equation 6} and either $\mathbf{(H_1)}$ or $\mathbf{(H_2)}$ hold, then setting $$m_n=k_nb_nv-\frac{3}{2\theta^{*}}\log(n)+\frac{\log(b_n)}{\theta^{*}},$$  the extremal process $$\mathcal{E}^{(b_n)}_n=\sum_{u\in{\mathcal{H}_{k_n}}}\delta_{S^{(n)}_u-m_n}$$  converges in law to a Poisson point process with intensity $\frac{1}{\sqrt{2\pi\sigma^2}}Z_{\infty}e^{-\theta^{*}x}$, where $\sigma^2=\kappa^{''}_n(\theta^*)$ and $Z_{\infty}$ is the random variable defined in equation \eqref{500}.  Moreover, the law of the recentered maximum converges weakly to a Gumbel distribution randomly shifted by $\frac{1}{\theta^*}\log(Z_{\infty})$. 
	\end{theorem}
	\begin{remark}
	\label{remark} Denote by $\mathcal{C}^{l,+}_{b}$ the set of continuous, positive and  bounded functions $\phi:\mathbb{R}\rightarrow \mathbb{R_{+}}$ with support bounded on the left. By \cite[Lemma 4.1]{berestycki2018extremes}, it is enough to show that for all  function $\phi \in{\mathcal{C}^{l,+}_{b}}$ $$\lim_{n\to \infty}\mathbb{E}\left(e^{-\sum_{u\in{\mathcal{H}_{k_n}}} \phi(S_{u}^{(n)}-m_n)}\right)=\mathbb{E}\left(\exp\left(-Z_{\infty}\frac{1}{\sqrt{2\pi\sigma^2}}\int_{}e^{-\theta^{*}y}(1-e^{-\phi(y)})dy\right)\right).$$
	The result of Kistler and Schmidt \cite[Theorem 1.1]{zbMATH06473046} is covered by Theorem $1$. It is the case $\mathbf{(H_1)}$ with $k_n=N^{\alpha}$, $0\leq\alpha<1$ and $Z_{1}= 2$  in our theorem. In that case we have $Z_{\infty}=1$ and $m_n=n\beta_c-\frac{2\alpha+1}{2\beta_c}\log(n)$.
	Throughout this paper, we use $C$ and $c$ to denote  generic positive constants, that may change from line to line. We say that $f_n\sim_{ n \to \infty}g_n$ if $ \lim_{n\to\infty}\frac{f_n}{g_n}=1$. For $x\in{\mathbb{R}}$ we write $x_+=\max(x,0)$.
	\end{remark} 
	 The rest of the paper is organized as follows. In the next section,  we introduce the many to one lemma, and we will give a series of useful random walk estimates.  In Section $4$ we introduce a modified extremal process which we show to have same  asymptotic behaviour of the original extremal process defined in the principal theorem. Finally we will conclude the paper with a proof of the main result.

	\section{Many-to-one formula and random walk estimates}
	In this section, we introduce the many-to-one lemma, that links
additive moments of branching processes to random walk estimates. We
then introduce some estimates for the asymptotic behaviour of random
walks conditioned to stay below a line, and prove their extension to a
generalized random walk where the law of each step is given by the sum of $b_n$ i.i.d random variables.

	\subsection{Many-to-one formula}
	We start by  introducing the celebrated many-to-one lemma that transforms an additive function of a branching random walk into a simple function of random walk. This lemma was introduced by Kahane and Peyrière \cite{zbMATH03544943}. Before we introduce it,  we need to define some change of measure and to introduce some notation. 
	
	Let $W_0:=0$ and $(W_{j}-W_{j-1})_{j\geq 1}$ be a sequence of independent and identically distributed  random variables such that for any measurable function $h:\mathbb{R}\mapsto\mathbb{R}$,
	\begin{align*}
	    \mathbb{E}(h(W_1))=\mathbb{E}\left(e^{\theta^*Y_1-\Lambda(\theta^*)}h(Y_1)\right).
	 \end{align*}
	 where $Y_1$ is the law defined in Section $2$.
	 Respectively, we introduce  $(T^{(n)}_{j}-T^{(n)}_{j-1})_{j\geq 1}$ a sequence of i.i.d random variables such that $T_0=0$ and 
	\begin{align} \label{mh} \mathbb{E}(h(T^{(n)}_1))=\frac{\mathbb{E}\left(\sum_{u,|u|=1}e^{\theta^{*} S^{(n)}_u}h(S^{(n)}_u)\right)}{\mathbb{E}(\sum_{u,|u|=1}e^{\theta^{*} S^{(n)}_u})}=\mathbb{E}\left(e^{\theta^*Y_{b_n}-\Lambda(\theta^*})h(Y_{b_n})\right).
	\end{align} 
	 
	 Observe that 
	 $(T^{(n)}_{k},k\geq1)$
	 is a sequence of random variables that have the same law as the process $(U_{kb_n}=\sum_{j=1}^{kb_n}W_j ,k\geq 1)$.
	 We now set $\bar{T}^{(n)}_{j}=T^{(n)}_{j}-jb_nv$ respectively $\bar{W}_{j}=W_{j}-jv, j\geq 1 $.  
	We have  $$\mathbb{E}(W_1)=\mathbb{E}\left(Y_{1}e^{\theta^*Y_{1}-\Lambda(\theta^{*})}\right)=\Lambda^{'}({\theta^*}),$$ and as $\Lambda^{'}({\theta^*})=\kappa^{'}_{n}(\theta^*)=v$, we have $\mathbb{E}(\bar{W}_{1})=0$ and similarly
	 \begin{align*}
	   &\mathbb{E}\left(W^{2}_1\right)=\mathbb{E}\left(Y^{2}_1 e^{\theta^*Y_{1}-\Lambda(\theta^{*})}\right)=\Lambda^{''}(\theta^*)+(\Lambda^{'}(\theta^{*}))^{2},
	 \end{align*}
	 which gives $\mathrm{Var}(\bar{W}_1)=\Lambda^{''}(\theta^{*})=\sigma^2$ which is finite by assumption \eqref{equation 6}.  As a consequence we have $\mathbb{E}(\bar{T}^{(n)}_1)=0$ and $\mathrm{Var}(\bar{T}^{(n)}_1))=b_n\sigma^2<\infty.$
	 In the case $(\mathbf{H}_1)$, note that $\bar{W}_1$ is a standard Gaussian random variable which means that $\bar{T}^{(n)}_1$ is a centred Gaussian random variable with variance $b_n$.
	 
 For simplicity we write $S_u$ in place of $S^{(n)}_{u}$ and $T_j$ in place of $T^{(n)}_{j}$ in the rest of the article.
		\begin{proposition}\cite[Theorem 1.1]{zbMATH06492274}\label{M} For any $j\geq 1$ and any measurable function $g:\mathbb{R}^{j}\to\mathbb{R_+}$, we have $$\mathbb{E}\left(\sum_{|u|=j}g((S_{u_{i}})_{1\leq i\leq j})\right)=\mathbb{E}\left(e^{-\theta^{*}\bar{T_i}}g((\bar{T_i}+ib_nv)_{1\leq i\leq j})\right).$$  
			\end{proposition}
		\begin{proof}
			For $j=1$, by $\eqref{mh}$ and using that $b_nv=\frac{\kappa_n{(\theta^*)}}{\theta^*},$ we have
			$$\mathbb{E}\left(\sum_{|u|=1}g(S_{u})\right)=\mathbb{E}(e^{-\theta^{*}T_1+\kappa_n(\theta^{*})}g(T_1))=\mathbb{E}\left(e^{-\theta^{*}\bar{T_1}}g(\bar{T_1}+b_nv)\right)$$
			where $\bar{T}_{1}=T_{1}-b_nv$.
			 We complete the proof by induction in the the same way as in \cite[Theorem 1.1]{zbMATH06492274}.
		\end{proof}
	\subsection{Random walk estimates} 
	In this section we introduce 
	some estimates for the asymptotic behaviour of functionals of the random walks, such us the probability to stay above a boundary.
We first give an estimate for the probability that a random walk stays above a boundary $(f_n)_{n\in{\mathbb{N}}}$, that is $O(n^{1/2-\epsilon})$ for some $\epsilon>0$.  %
	\begin{lemma} \cite[Lemma 3.6]{zbMATH06471546}.
		\label{MA}
		Let $(w_n)_{n\in\mathbb{N}}$ be a centred random walk with finite variance. Fix $\epsilon>0$, there exists $C>0$ such that  $$ \mathbb{P}(w_k \geq -(k^{1/2-\epsilon}+y),k\leq n)\leq C\frac{1+y}{\sqrt{n}}$$ for any $y>0$.
	\end{lemma}
	From now on we  use the random walks $(T_k)_{k\geq1}$ and $(\bar{T}_k)_{k\geq1}$ defined in \eqref{mh}, unless otherwise stated. In the next lemma we will give an approximation of the probability for a random walk to end up in a finite interval using the Stone's local limit theorem \cite{zbMATH03374724}.
	\begin{lemma}
		\label{Stone}
		Let $f\in{\mathcal{C}^{l,+}_b}$  be a Riemann integrable function, and let $(r_n)_{n\in{\mathbb{N}}}$  be a sequence of positive real numbers, such that $\lim_{n\to \infty }\frac{r_n}{\sqrt{n}}=0$. Set  $$a_n=\frac{-3}{2\theta^{*}}\log(n)+\frac{\log(b_n)}{\theta^{*}}$$ then we get $$
		 \mathbb{E}(f(\bar{T}_{k_n}-a_n+x)e^{-\theta^{*}\bar{T}_{k_n}})=\frac{e^{\theta^*x}n^{3/2}}{b_n\sqrt{2\pi\sigma^2 
		 k_nb_n}}\int f(y)e^{-\theta^{*}y} dy(1+o(1))$$
		uniformly in $x\in{[-r_n, r_n]}$.
	\end{lemma}
\begin{proof}
	 By setting $h(z)=e^{-\theta^* z}f(z),$
	 it is enough to prove that \begin{align}
	 \label{501}
	 &\mathbb{E}(h(\bar{T}_{k_n}-a_n+x))=
	 \frac{1}{\sqrt{2\pi\sigma^2
	 k_nb_n}}\int h(y) dy(1+o(1))
	 \end{align}
	 uniformly in $x\in{[-r_n, r_n]}$. 
	   We  prove this lemma by successive approximations of the function $h$, 
	  starting with an indicator function. Set $h(z) =\mathbf{1}_{[a,b]}(z)$ for some $a<b\in{\mathbb{R}}$, then we write 
	\begin{align}
	\label{eee}
	&\mathbb{E}(h(\bar{T}_{k_n}-a_n+x))=\mathbb{P}\left(\bar{T}_{k_n}-a_n+x\in[a,b]\right),
	\end{align}	
As $\bar{T}_1$ is the sum of $b_n$ i.i.d. copies of
$\bar{Z}_1$, $\bar{T}_{k_n}$ is the sum of $k_n b_n$ i.i.d. centred
random variables with finite variance, therefore we can apply the  Stone's local limit theorem \cite{zbMATH03374724} to obtain 
	\begin{align*}
	&\mathbb{P}(\bar{T}_{k_n}-a_n+x\in[a,b])=\frac{b-a}{\sqrt{2\pi \sigma^2
	k_nb_n}}\exp\left(\frac{-(a_n-x)^2}{2k_nb_n\sigma^2}\right)(1+o(1)) =\frac{b-a}{\sqrt{2\pi  k_nb_n\sigma^2}}(1+o(1)),
	\end{align*} 
	uniformly in $x\in[-r_n,r_n]$, which completes the proof of \eqref{501} in that case.
	
 We now assume that $h$ is a continuous function  with compact support, we prove \eqref{501} by approximating it by simple functions. Denote by $[a,b]$ the support of $h$. Let $(t_i)_{0\leq i\leq m}$ be an uniform subdivision  of  $[a,b]$ where $m\in{\mathbb{N}}$ is the number of the subdivisions and  $t_i= a+i(b-a)/m$ for $0\leq i\leq m$. Set 
	$$\underline{h}_m(x) =\sum_{i=0}^{m-1} m_i\mathbf{1}_{\left\{x \in [t_i,t_{i+1}]\right\}}
\hspace{0.3cm}\text{and}\hspace{0.3cm}  \bar{h}_m(x)=\sum_{i=0}^{m-1} M_i\mathbf{1}_{\left\{x \in{ [t_i,t_{i+1}]}\right\}},$$
	 where $M_i=\sup_{z \in [t_i,t_{i+1}]}h(z)$ and $m_i=\inf_{z \in [t_i, t_{i+1}]}h(z)$.
	Hence  using the Riemann sum approximation and the fact that $f$ is a non-negative function, %
	for all $\epsilon>0$, there exists $m_0$ such that for all $m\geq m_0$ we have
	\begin{align}
	\label{approx}
	&(1-\epsilon)\int_{a}^{b} h(y)dy \leq \int_{a}^{b} \underline{h}_m(y)dy \leq \int_{a}^{b} \bar{h}_m(y)dy \leq
	(1+\epsilon)\int_{a}^{b} h(y)dy,
	\end{align}
	where $\int_{a}^{b} \underline{h}_m(y)dy=\sum_{i=0}^{m-1} \frac{b-a}{m}m_i$ and  $\int_{a}^{b} \bar{h}_m(y)dy=\sum_{i=0}^{m-1} \frac{b-a}{m}M_i$.
	
	Using equation \eqref{eee} we have
	\begin{align*}
&\mathbb{E}\left(\bar{h}_m(\bar{T}_{k_n}-a_n+x)\right)=\sum_{i=0}^{m-1}M_i\mathbb{P}\left(\bar{T}_{k_n}-a_n+x\in{[t_i, t_{i+1}[}\right)\displaystyle=\frac{1}{\sqrt{2\pi \sigma^2
k_nb_n}}\sum_{i=0}^{m-1} \frac{b-a}{m}M_i(1+o(1))\\&=\frac{1}{\sqrt{2\pi  \sigma^2
k_nb_n}}\int_{a}^{b} \bar{h}_m(y)dy(1+o(1)).
\end{align*}
Therefore, using that
$\mathbb{E}(h(\bar{T}_k -a_n + x) \leq
\mathbb{E}(\bar{h}_m(\bar{T_k} - a_n + x)$ and by \eqref{approx} we deduce that
\begin{align*}
&\limsup_{n\to\infty}\sup_{x\in{[0,r_n]}}\sqrt{k_n b_n}\mathbb{E}\left(h(\bar{T}_{k_n}-a_n+x)\right)\leq (1+\epsilon)\frac{1}{\sqrt{2\pi}
\sigma^2}
\int_{a}^{b}h(y)dy.
\end{align*}
Using similar arguments we have 
\begin{align*}
&\liminf_{n\to\infty}\inf_{x\in{[0,r_n]}}\sqrt{k_nb_n}\mathbb{E}\left(h(\bar{T}_{k_n}-a_n+x)\right)\geq (1-\epsilon)\frac{1}{\sqrt{2\pi \sigma^2
}}\int_{a}^{b}h(y)dy.
\end{align*}
Finally, letting $\epsilon \to 0$ completes the proof of \eqref{501} when $h$ is
a compactly support function.
	Finally we consider the general case, and assume that $f$ is bounded
with bounded support on the left. We introduce the function

	\begin{equation}
	\nonumber
	\chi(u)=
	\left\lbrace
	\begin{array}{ccc}
	1 & \mbox{if} & u<0\\
	1-u & \mbox{if} & 0\leq u\leq 1\\
	0 & \mbox{if} & u>1
	\end{array}\right.
	\end{equation} then we write,
	\begin{align*}
	&\mathbb{E}(h(\bar{T}_{k_n}-a_n+x))=\mathbb{E}\left(h(\bar{T}_{k_n}-a_n+x)\chi(\bar{T}_{k_n}-a_n+x-B)\right)\\&+\mathbb{E}\left(h(\bar{T}_{k_n}-a_n+x)(1-\chi(\bar{T}_{k_n}-a_n+x-B))\right)
	\end{align*} 
	for some $B>0$. Observe that the function $z\mapsto h(z)\chi(z-B)$ is continuous with compact support, then using previous result, we have 	
	\begin{align}
	   \label{Local}
	&\nonumber\mathbb{E}(h(\bar{T}_{k_n}-a_n+x))=\frac{1}{\sqrt{2\pi \sigma^2
	k_nb_n}}\int h(y)\chi(y-B)dy(1+o(1))\\+&\mathbb{E}\left(h(\bar{T}_{k_n}-a_n+x)(1-\chi(\bar{T}_{k_n}-a_n+x-B))\right). 
	\end{align}
	 Thanks to the Stone's local limit theorem \cite{zbMATH03374724} there exists a constant $C>0$ such that  the  quantity  \eqref{Local} is bounded by 
	\begin{align*}
	&\mathbb{E}\left(h(\bar{T}_{k_n}-a_n+x)(1-\chi(\bar{T}_{k_n}-a_n+x-B))\right)\leq \mathbb{E}\left(h(\bar{T}_{k_n}-a_n+x )\mathbf{1}_{\left\{\bar{T}_{k_n}-a_n+x>B\right\}}\right)\\&\leq||f||_{\infty}\mathbb{E}\left(\sum_{j\geq B}e^{-\theta^*j}\mathbf{1}_{\left\{\bar{T}_{k_n}-a_n+x\in{[j,j+1]}\right\}}\right)\leq C ||f||_{\infty}\frac{e^{-\theta^*B}}{\sqrt{k_nb_n \sigma^2
	}}.
	\end{align*}
	On the other hand by the dominated convergence theorem we have  
	\begin{align*}
	&\lim_{B\to\infty}\frac{1}{\sqrt{2\pi\sigma^2 k_nb_n}}\int h(y)\chi(y-B)dy= \frac{1}{\sqrt{2\pi \sigma^2
	k_nb_n}}\int h(y)dy,
	\end{align*}
   as a consequence we deduce that  $$\mathbb{E}(h(\bar{T}_{k_n}-a_n+x))=
 \frac{1}{\sqrt{2\pi \sigma^2
 k_nb_n}}\int f(y)e^{-\theta^{*}y}dy(1+o(1)),$$
 which completes the proof.
\end{proof}
	\subsubsection{Random walk with Gaussian steps}
	In this section we assume that $\mathbf{(H_1)}$ holds, i.e that $(\bar{T_k})_{k\geq 0}$ is a Gaussian random walk. Let $(\beta_n(k),k\leq k_n)$  be the standard discrete Brownian bridge with $k_n$ steps, which can be defined as,
	$$\beta_n(k)=\frac{1}{\sqrt{b_n}}(\bar{T}_{k}-\frac{k}{k_n}\bar{T}_{k_n}).$$
	 In the following lemma we estimate the probability for a Brownian bridge to stay below a boundary during all his lifespan.
	 This lemma was introduced in \cite[proposition $1$]{zbMATH03562205} for continuous time Brownian motion  which also hold for the discrete time version. 
	\begin{lemma}
		\label{4}
	Let $h$ be the function defined by 
		$$
		h(k) = \left\{
		\begin{array}{ll}
		0 & \mbox{if } k =0 \hspace{0.1cm} or\hspace{0.1cm} k=k_n \\
		a\log((k_n-k)\land k)b_n)+
		1) & \mbox{otherwise.}
		\end{array}
		\right.
		$$
	 where $a$ is a positive constant.
		 There exists a constant $C>0$ such that for all $x>0$ and $n\geq 0$ we have
		\begin{align}
		&\mathbb{P}\left(
	\beta_n(k)\leq\frac{1}{\sqrt{b_n}}(h(k)+x),k\leq k_n \right)\leq C\frac{(1+\frac{x}{\sqrt{b_n}})^{2}}{k_n}.
		\end{align}
	\end{lemma}
 We refer to the function $k\mapsto h(k)$ as a barrier.
An application of this lemma is to give an upper bound for the probability that a random walk with Gaussian steps make an excursion above a well-chosen barrier.
		\begin{lemma}
		\label{10}
		Let $\alpha>0$, and for $0\leq k\leq k_n$ we write $f_n(k)=\alpha \log(\frac{(k_n-k)b_n+1}{k_nb_n})$. There exists $C>0$ such that for all $x \geq 0$, $a<b\in{\mathbb{R}}$ and $k\leq k_n$ we have 
		\begin{align*}
		&\mathbb{P}\left(\bar{T_k}-f_n(k)\in[a,b], \bar{T}_j\leq f_n(j)+x, j\leq k\right) \leq C(b-a)\frac{(1+\frac{x}{\sqrt{b_n}})^{2}}{\sqrt{
				b_n}k^{\frac{3}{2}}}.
		\end{align*}
	\end{lemma}
\begin{proof}
For $n\in{\mathbb{N}}$  we have 
\begin{align*}
&\mathbb{P}\left(\bar{T_k}-f_n(k)\in[a,b],\bar{T}_j\leq f_n(j)+x, j\leq k\right)\\&\leq \mathbb{P}\left(\bar{T}_k-f_n(k)\in{[a,b]},\bar{T}_j-\frac{j}{k}\bar{T}_k\leq f_n(j)+x-\frac{j}{k}(f_n(k)+a), j\leq k\right),
\end{align*}   
using independence between the discrete Brownian bridge  $\bar{T}_j-\frac{j}{k}\bar{T}_k$ and $\bar{T_k}$ we obtain 
\begin{align}
&
\label{9}
\mathbb{P}\left(\bar{T_k}-f_n(k)\in[a,b],\bar{T}_j-\frac{j}{k}\bar{T}_k\leq f_n+x-\frac{j}{k}(f_n(k)), j\leq k\right)\\& \nonumber \label{6}\leq \mathbb{P}\left(\bar{T}_k-f_n(k)\in [a,b]\right) \mathbb{P}\left( \bar{T}_j-\frac{j}{k}\bar{T}_k\leq f_n(j)+x-\frac{j}{k}(f_n(k)+a), j\leq k\right).
\end{align}
To estimate the probability that a discrete Brownian bridge stay below a logarithmic barrier, we apply Lemma \ref{4}. First observe that the function $ x\mapsto \frac{\log(x)}{x}$ is decreasing for $x\geq e$, and using that $(k_n-j)b_n+1\leq (k_n-k)b_n+1+(k-j)b_n+1\leq 2((k_n-k)b_n+1)(k-j)b_n+1))$, we have for $j\leq\frac{k}{2}$,
\begin{eqnarray*}
	f_n(j)+x-\frac{j}{k}(f_n(k)+a)&\leq& 
	\alpha\frac{j}{k}\left(\log(\frac{k_nb_n}{(k_n-k)b_n+1})-\log(\frac{k_nb_n}{(k_n-j)b_n+1})\right)+x\\&\leq &\alpha\frac{j}{k}(\log(kb_n)+\log(2))+x \leq\alpha(\log((jb_n\lor e))+\log(2))+x
\end{eqnarray*}
and for $\frac{k}{2}\leq j \leq k$, we have
\begin{eqnarray*}
	f_n(j)+x-\frac{j}{k}(f_n(k)+a)&\leq &  \alpha(\log(\frac{k_nb_n}{(k_n-k)b_n+1})+x-\log(\frac{k_nb_n}{(k_n-j)b_n+1}))\\&\leq& \alpha(\log(((k_n-j)b_n+1)-\log((k_n-k)b_n+1)))+x\\&\leq& \alpha(\log(2)+\log(1+(k-j)b_n)+x.
\end{eqnarray*}
Then by Lemma \ref{4} we get after rescaling by $\frac{1}{\sqrt{b_n}}$ the following upper bound
\begin{align*}
&\mathbb{P}\left( \bar{T}_j-\frac{j}{k}\bar{T}_{k}\leq f_n(k)-\frac{j}{k}(f_n(j)-x), j\leq k\right)\\&\leq  \mathbb{P}\left(
\beta_n(k) \leq \alpha(\log((k\land(k-j))+1))+\frac{x}{\sqrt{b_n}}+1, j\leq k\right)\leq C\frac{(1+\frac{x}{\sqrt{b_n}})^{2}}{k},
\end{align*}
where $C$ is a positive constant. To bound the first quantity in \eqref{9} we use the Gaussian estimate
\begin{align*}
&\mathbb{P}\left(\bar{T_k}-f_n(k)\in[a,b]\right)\leq \frac{b-a}{\sqrt{kb_n
}}
\end{align*}
which completes the proof.	
	\end{proof}
	   	From now  we denote by $B_n(k)=\frac{\bar{T}_k}{\sqrt{b_n}}$. Recall that under $(\mathbf{H}_1),$ $(B_n(k))_{k\leq k_n}$ is a standard random walk with i.i.d Gaussian steps.
	    	Define the function $L:(0,\infty)\to(0,\infty)$ by $L(0)=1$ and 
	\begin{align}
	   & L(x):=\sum_{k\geq 0}\mathbf{P}\left(B_n(k) \geq -x, B_n(k) \leq \min_{j\leq k-1} B_n(j)
	\right) \hspace{0.5cm} \text{ for } x>0.
	\end{align}
	It is known by \cite[section XII.7] {feller1971introduction}, that the function $L$ is the renewal function associated to the random walk $(B_n(k))_{k\geq 0}$. We will cite some properties that are mentioned  in \cite[section XII.7] {feller1971introduction}. The fundamental property of the renewal function is  
	\begin{align}
	    \label{mhhh}
	    &L(x)=\mathbb{E}\left(L(x+B_n(1)
	    )\mathbf{1}_{\left\{x+B_n(1)
	    \geq 0\right\}}\right),
	\end{align}
	and is a  right-continuous and non-decreasing function. Since in case $(\mathbf{H_1})$, the initial law has no atoms, then the function $L$ is continuous.
   Also, there exists a constant $c_0>0$ such that
	\begin{align}
	\label{565}
		&\lim_{x\to\infty}\frac{L(x)}{x}=c_{0}.
	\end{align}
	In particular there exists a constant $C>0$  such that for all $x\in {\mathbb{R}}$
	\begin{align}
	\label{renewl}
	    &L(x)\leq C (1+x_{+}).
	\end{align}
Also we have by \cite[section XII.7] {feller1971introduction} , for $x,y\geq 0$
\begin{align}
\label{1v1}
    &L(x+y)\leq 2 L(x) L(y) .
\end{align}
	Similarly, we define $L_{-}(x)$ as the renewal function associated to $-B$.
	
	 Since $\bar{T}$ is a symmetric law we have $L_{-}(x)=L(x)$ for all $x\geq 0$. 
	It is also known  that there exists a positive constant $C_1$ such that for $y\geq 0$ 
	\begin{align}
	\label{renewal}
		&\mathbb{P}\left(\min_{k\leq k_n}(B_n(k)
		)\geq -y\right)\sim_{n\to\infty}C_1\frac{L(y)}{\sqrt{k_n}}.
	\end{align}
	 By Theorem 3.5 in \cite{spitzer1960tauberian}, assuming that $B$ is Gausian  we have $C_1=\frac{1}{\sqrt{\pi}}$. 
	We now introduce 
	an approximation of the probability for a random walk to stay below a line and end up in a finite interval .
	  Set $$\tilde{F}_n(k)=\frac{k}{k_n}a_n=\frac{k}{k_n}(m_n-k_nb_nv),  k=0..., k_n, \hspace{0.2cm} n\in{\mathbb{N}}.$$
	\begin{lemma}
		\label{Stone barrier} 
		Let $(r_n)_{n\in{\mathbb{N}}}$ be a sequence of positive
real numbers such that $\lim_{n \to \infty} \frac{r_n}{\sqrt{k_n}} = 0$. Let  $$a_n=\frac{-3}{2\theta^{*}}\log(n)+\frac{\log(b_n)}{\theta^{*}}.$$ 
For all $f \in \mathcal{C}^{l,+}_b$ we have
		 
		$$\mathbb{E}\left(f(\bar{T}_{k_n}-a_n+x)e^{-\theta^{*}\bar{T}_{k_n}}\mathbf{1}_{\{\bar{T}_{k}\leq\bar{F}_n(k)-x,k\leq k_n\}}\right)=\frac{e^{\theta^{*}x}}{\sqrt{2
		\pi}}\int_{-\infty}^{0} f(y)e^{-\theta^{*}y}dy\left(R(\frac{-x}{\sqrt{b_n}})+o(1)\right).$$
		uniformly in $x \in{[-r_n, 0]}.$
	\end{lemma}
	
	\begin{proof}
			By setting $h(z)=e^{-\theta^* z}f(z)$
		it is enough to prove that 
		\begin{align}
		\label{function estimate}
		&\mathbb{E}\left(h(\bar{T}_{k_n}-a_n+x)\mathbf{1}_{\{\bar{T}_{k}\leq\bar{F}_n(k)-x,k\leq k_n\}}\right)=
		\frac{1}{k_n^{3/2}\sqrt{2
		\pi b_n}}\int_{-\infty}^{0} h(y)dy(	R(\frac{-x}{\sqrt{b_n}})+o(1))
		\end{align}
		uniformly in $x\in{[-r_n,0]}$.
	
			Following the same method used in Lemma \ref{Stone} it is enough to prove this estimate for an indicator function.
	By writing $\mathbf{1}_{[-a,-b]}=\mathbf{1}_{[-a,0]}-\mathbf{1}_{[-b,0]}$ for some $a>0,b>0$, it is enough to prove this estimate for 
		  $h(z) =\mathbf{1}_{[-a,0]}(z)$, in that case we have
		  \begin{align*}
		&\mathbb{E}\left(h(\bar{T}_{k_n}-a_n+x)\mathbf{1}_{\{\bar{T}_{k}\leq \bar{F}_n(k)-x,k\leq k_n\}}\right)=\mathbb{P}\left(\bar{T}_{k_n}-a_n+x\geq -a, \bar{T}_{k}\leq F_n(k)-x,k\leq k_n\right).
		\end{align*}
			Define a new probability measure $\mathbb{Q}$ on $\mathbb{R}$ by 
			\begin{align}
			\label{change}
			  &  \frac{d\mathbb{P}}{d\mathbb{Q}}(\bar{T})=\exp(\frac{-a_n}{n}\bar{T}+\Lambda(\frac{a_n}{n}))
			\end{align}
	where $\Lambda(\theta)=\frac{\theta^2}{2}.$ 
	      Then we rewrite 
		\begin{align*}
		&\mathbb{P}\left(\bar{T}_{k_n}-a_n+x\geq -a, \bar{T}_{k}\leq F_n(k)-x,k\leq k_n\right)\\&=	\mathbb{E}_{Q}\left(e^{\frac{-a_n}{n}(\sqrt{b_n}\hat{B}_{n}(k_n)-\frac{a_n}{2n})}\mathbf{1}_{\{\sqrt{b_n}\hat{B}_{n}(k_n)+x\geq -a,\sqrt{b_n}\hat{B}_{n}(k)\leq -x,k\leq k_n\}}\right),
		\end{align*}
		where $\hat{B}_n(k)=B_{n}(k)-\frac{k}{\sqrt{b_n}k_n}a_n$.
	     Observe that 
	     the law of $\hat{T}$ under $\mathbb{Q}$ is the same as
	     the law of $\bar{T}$ under $\mathbb{P}$.
	     
		 Under this change of measure, we can rewrite the probability as  
		\begin{align*}
		&\mathbb{E}_{Q}\left(e^{-\frac{a_n}{n}\sqrt{b_n}\hat{B}_{n}(k_n)+\frac{a_n^2}{2n}}\mathbf{1}_{\left\{\sqrt{b_n}\hat{B}_{n}(k)\leq-x,k\leq k_n,\sqrt{b_n}\hat{B}_n(k_n)+x\geq -a\right\}}\right)\\&\leq e^{\frac{a_n}{n}(x+a)+\frac{a_n^2}{2n}} \mathbb{Q}\left(\sqrt{b_n}\hat{B}_n(k)\leq-x,k\leq k_n,\sqrt{b_n}\hat{B}_n(k_n)\geq -a-x\right).
		\end{align*}
		as a consequence   
		\begin{align*}
		    &\limsup_{n\to \infty}\sup_{x\in{[-r_n,0]}}\mathbb{E}_{Q}\left(e^{-\frac{a_n}{n}\sqrt{b_n}\hat{B}_{n}(k_n)+\frac{a_n^2}{2n}}\mathbf{1}_{\left\{\sqrt{b_n}\hat{B}_{n}(k)\leq-x,k\leq k_n,\sqrt{b_n}\hat{B}_{n}(k_n)+x\geq -a\right\}}\right)\\&\leq \limsup_{n\to \infty}\sup_{x\in{[-r_n,0]}} \mathbb{Q}\left(\sqrt{b_n}\hat{B}_{n}(k)\leq-x,k\leq k_n,\sqrt{b_n}\hat{B}_{n}(k)\geq -a-x\right),
		\end{align*}
	similarly we have
		\begin{align}
		    &\nonumber\liminf_{n\to \infty}\inf_{x\in{[-r_n,0]}}\mathbb{E}_{Q}\left(e^{-\frac{a_n}{n}\sqrt{b_n}\hat{B}_{n}(k_n)+\frac{a_n^2}{2n}}\mathbf{1}_{\left\{\sqrt{b_n}\hat{B}_{n}(k)\leq-x,k\leq k_n,\sqrt{b_n}\hat{B}_{n}(k_n)+x\geq -a\right\}}\right)\\&\label{eqq}\leq \liminf_{n\to \infty}\inf_{x\in{[-r_n,0]}} \mathbb{Q}\left(\sqrt{b_n}\hat{B}_{n}(k)\leq-x,k\leq k_n,\sqrt{b_n}\hat{B}_{n}(k_n)\geq -a-x\right).
		    \end{align}
		for all $a>0$. Therefore, it remains to estimate the quantity \eqref{eqq}.
		Applying the Markov property at time $p=[\frac{k_n}{2}]$ we get 
		\begin{align}
		\label{5006}
		&\mathbb{Q}\left(\sqrt{b_n}\hat{B}_{n}(k)\leq-x,k\leq k_n,\sqrt{b_n}\hat{B}_{n}(k_n)\geq -a-x\right)=\mathbb{E}\left( f_{x,n,a}(\sqrt{b_n}\hat{B}_{n}(p))\mathbf{1}_{\left\{\sqrt{b_n}\hat{B}_n(k)\leq -x,k\leq p\right\}}\right)  
		\end{align}
		where for all $y\leq 0$
		\begin{align*}
		&f_{x,n,a}(y)=\mathbb{Q}\left(\sqrt{b_n}\hat{B}_{n}(k_n-p)+y\geq -a-x ,\sqrt{b_n}\hat{B}_n(k)+y\leq -x, k\leq k_n-p\right).
		\end{align*}
		Using that the process $(\sqrt{b_n}(\hat{B}_{n}(k_n-p)-\hat{B}_n(k_n-p-j)),0\leq j\leq k_n-p)$ has the same law as $(\sqrt{b_n}\hat{B}_n(j),0\leq j\leq k_n-p)$ under $\mathbb{Q}$, we obtain 
		\begin{align*}
		&f_{x,n,a}(y)=\mathbb{Q}\left((-\sqrt{b_n}\hat{B}_n(k)) \leq (-\sqrt{b_n}\hat{B}_{k_n-p})-(x+y)\leq a, k\leq k_n-p\right)\\&=\mathbb{Q}\left(\sqrt{b_n}\hat{B}_n(k) \leq \sqrt{b_n}\hat{B}_n(k_n-p)-(x+y)\leq a, k\leq k_n-p\right)
		\end{align*}
		since $(\sqrt{b_n}\hat{B}_n(k))_{k\geq 0}$ is a symmetric law. We write $\check{B}_n(k_n-p)=\max_{0\leq j\leq k_n-p}\sqrt{b_n}\hat{B}_n(i),$ set
		 $$\tau_{k_n-p}=\min{\left\{i:\hspace{0.2cm}
			0\leq i\leq k_n-p,\hspace{0.2cm} \check{B}_n(k_n-p)=\sqrt{b_n}\hat{B}_n(i)\right\}}$$
 the first time when $\sqrt{b_n}\hat{B}_n(i)$ hits its maximum in the interval $[0,k_n-p].$
		We have 
		\begin{align*}
		&f_{x,n,a}(y)=\sum_{i=0}^{k_n-p}\mathbb{Q}\left(\tau_{k_n-p}=i,\sqrt{b_n}\hat{B}_n(k) \leq \sqrt{b_n}\hat{B}_n(k_n-p)-(x+y)\leq a, k\leq k_n-p \right).
		\end{align*}
		Applying the Markov property at time $i$ we get
		\begin{align*}
		&  f_{x,n,a}(y)=\sum_{i=0}^{k_n-p}\mathbb{E}\left( g_{x,n,y}\left(\check{B}_{n}(i)-a\right)\mathbf{1}_{\left\{\check{B}_n(i)=\sqrt{b_n}\hat{B}_n(i)\leq a \right\}}
		\right), 
		\end{align*}
		where for all $z\leq0$, 
		$g_{x,n,y}(z)=\mathbb{Q}\left(y+x\leq \sqrt{b_n}\hat{B}_n(k_n-p-i)\leq y+x-z,\check{B}_n(k_n-p-i)\leq 0\right)$.
		
		We now split the sum $\sum_{i=0}^{k_n-p}$ into $\sum_{i=0}^{i_n}+\sum_{i=i_n+1}^{k_n-p}$, where $i_n=[\sqrt{k_n}]$, then we write 
		\begin{align*}
		&f_{n,x,a}(y)=   f^{(1)}_{n,x,a}(y)+ f^{(2)}_{n,x,a}(y)
		\end{align*}
		where 
		$$ f^{(1)}_{n,x,a}(y)=\sum_{i=0}^{i_n}\mathbb{E}\left( g_{x,n,y}(\check{B}_n(i)-a)\mathbf{1}_{\left\{\check{B}_n(i)=\sqrt{b_n}\hat{B}_n(i)\leq a\right\}}\right),$$ and $$ f^{(2)}_{n,x,a}(y)=\sum_{i=i_n+1}^{k_n-p}\mathbb{E}\left( g_{x,n,y}(\check{B}_{n}(i)-a)\mathbf{1}_{\left\{\check{B}_n(i)=\sqrt{b_n}\hat{B}_n(i)\leq a\right\}}\right).$$
		Set $\phi(x):=xe^{\frac{-x^2}{2}}\mathbf{1}_{\left\{x\geq 0\right\}}$. By Theorem  $1$ in \cite{caravenna2005local} of Caravenna for $n \to \infty$,
		\begin{align*}
		& \mathbb{Q}\left(-(x+y-z)\leq \sqrt{b_n}\hat{B}_n(k_n-p-i)\leq -(x+y)|\sqrt{b_n}\hat{B}_n(j)\geq 0,j\leq k_n-p-i\right)\\&=\frac{-z}{\sqrt{(k_n-p)b_n}}\phi\left(\frac{-y}{\sqrt{(k_n-p)b_n}}\right)+o\left(\frac{1}{\sqrt{(k_n-p)b_n}}\right),
		\end{align*}
		uniformly in $y\leq 0$, $x\in{[-r_n,0]}$ and $z$ in any compact set of $\mathbb{R}_-$. As a consequence by \eqref{renewal} we get 
		\begin{align*}
		&g_{x,n,y}(z)=\frac{-z}{(k_n-p)\sqrt{b_n\pi}}\phi\left(\frac{-y}{\sqrt{(k_n-p)b_n}}\right)+o(\frac{1}{(k_n-p)\sqrt{b_n}}),
		\end{align*}
		uniformly in $y\leq 0$, $x\in{[-r_n,0]}$ and $z\in{[-a,0]}$.
		For $n$ large enough we get
		\begin{align}
		\label{zzz}
		&f^{(1)}_{n,x,a}(y)=\frac{1}{(k_n-p)\sqrt{b_n\pi}}\phi\left(\frac{-y}{\sqrt{(k_n-p)b_n}}\right)\sum_{i=0}^{i_n}\mathbb{E}\left(-(\hat{B}_n(i)-\frac{a}{\sqrt{b_n}})\mathbf{1}_{\left\{\hat{B}_n(k)\leq \frac{a}{\sqrt{b_n}},k\leq i \right\}}\right)\\&\nonumber+o(\frac{1}{k_n\sqrt{b_n}})\sum_{i=0}^{i_n}\mathbb{Q}\left(\hat{B}_n(k)\leq \frac{a}{\sqrt{b_n}}, k\leq i\right).
		\end{align}
		We  now treat the quantity 
		\begin{align*}
		&\mathbb{E}\left( f^{(2)}_{x,n,a}(\sqrt{b_n}\hat{B}_n(k_n-p))\mathbf{1}_{\left\{\sqrt{b_n}\hat{B}_n(k)\leq -x ,k\leq k_n-p\right\}}\right).
		\end{align*}
		Since $\phi$ is bounded, there exists a constant $C>0$ such that for all $x\in{[-r_n,0]},
		z\in{\mathbb{R}}$ and $0\leq i\leq p$ 
		\begin{align*}
		&g_{x,n,y}(z)\leq \frac{C}{\sqrt{b_n}(k_n-p-i+1)}\mathbf{1}_{\left\{-a \leq z\leq 0\right\}},
		\end{align*}
		as a consequence, for all $y\leq 0$ we have
		\begin{align*}
		&  f^{(2)}_{x,n,a}(y)\leq \frac{C}{\sqrt{b_n}} \sum_{i=i_n+1}^{k_n-p}\frac{1}{k_n-p-i+1}\mathbb{P}\left(\check{B}_n(i)\leq \frac{a}{\sqrt{b_n}},\hat{B}_n(i)\geq 0 \right)  
		\end{align*}
		which is bounded using Lemma \ref{10} by
		\begin{align}
		\label{equation1500}
		&f^{(2)}_{x,n,a}(y)\leq  \frac{C}{\sqrt{b_n}} \sum_{i=i_n+1}^{k_n-p}\frac{1}{(k_n-p-i+1)i^{\frac{3}{2}}}=o(\frac{1}{k_n\sqrt{b_n}}).
		\end{align}
		On the other hand  by \eqref{renewal} we have  $\mathbb{Q}(\hat{B}_n(j)\leq\frac{-x}{\sqrt{b_n}}, j\leq k_n-p)\sim_{n\to \infty}\frac{\sqrt{2}}{\sqrt{\pi}}\frac{L(\frac{-x}{\sqrt{b_n}})}{\sqrt{k_n}}$ by
		 \eqref{renewal}, which implies together with equation \eqref{equation1500} 
		
			\begin{align}
			\label{medd}
		&  (L(\frac{-x}{\sqrt{b_n}}))^{-1} \mathbb{E}\left( f^{(2)}_{x,n,a}(\hat{B}_n(k_n-p))\mathbf{1}_{\left\{\sqrt{b_n}\hat{B}_n(k)\leq -x ,k\leq k_n-p\right\}}\right)=o(\frac{1}{k_n^{\frac{3}{2}}\sqrt{b_n}}).
		\end{align}
		We now return to equation \eqref{zzz}. Letting $n \to \infty$,
		we have  $\sum_{i=0}^{\infty}\mathbb{P}\left(\frac{\hat{T}_i}{\sqrt{b_n}}\geq \frac{-a}{\sqrt{b_n}} \right)=R(\frac{a}{\sqrt{b_n}})$ for $a>0$ and by Fubini's theorem we have 
		\begin{align*}
		&  \sum_{i=0}^{\infty}\mathbb{E}\left((\hat{B}_n(i)+\frac{a}{\sqrt{b_n}})\mathbf{1}_{\left\{\hat{B}_n(k)\geq \frac{-a}{\sqrt{b_n}}, k\leq i \right\}}\right)=\int_{0}^{\frac{a}{\sqrt{b_n}}}L(t) dt=\frac{1}{\sqrt{b_n}}\int_{-a}^{0} L(\frac{-t}{\sqrt{b_n}})dt.
	\end{align*}
	 	By dominated convergence theorem
		we have
		\begin{align*}
		&\frac{1}{\sqrt{b_n}}\int_{-a}^{0} L(\frac{-t}{\sqrt{b_n}})dt=\frac{a}{\sqrt{b_n}}(1+o(1))=\frac{1}{\sqrt{b_n}}\int h(t) dt(1+o(1).
		\end{align*}
		since $h(t) =\mathbf{1}_{[-a,0]}(t)$.
		This yields, for all $y\leq 0$
		\begin{align*}
		& f^{(1)}_{n,x,a}(y)=\frac{1}{k_n\sqrt{b_n\pi}}\phi(\frac{-y}{\sqrt{(k_n-p)b_n}}) \int h(t )dt+o(\frac{1}{k_n\sqrt{b_n}}).
		\end{align*}
		as a consequence, by \eqref{renewal} 
		\begin{align*}
		&\mathbb{E}\left( f^{(1)}_{x,n,a}(\sqrt{b_n}\hat{B}_n(k_n-p))\mathbf{1}_{\left\{\hat{B}_n(k)\leq \frac{-x}{\sqrt{b_n}},k\leq k_n-p\right\}}\right)=\frac{2\sqrt{2}}{\pi k_n^{\frac{3}{2}}\sqrt{b_n}} \int h(t)dt\\&\times  \mathbb{E}_{\frac{-x}{\sqrt{b_n}}}\left(\phi\left(\frac{-\hat{B}_n(k_n-p)+\frac{x}{\sqrt{b_n}}}{\sqrt{k_n-p}}\right)\big\rvert -\hat{B}_n(k)\geq 0, k\leq k_n-p\right) \left(L(\frac{-x}{\sqrt{b_n}})+o(1)\right).
		\end{align*}
		On the other hand, it's known (see Lemma 2.2 in \cite{zbMATH05935920}) that under $\mathbb{P}_y\left(.\lvert\frac{\hat{T}_k}{\sqrt{b_n}}\geq 0, k\leq k_n-p\right)$, $\frac{\hat{B}_n(k_n)}{\sqrt{k_n}}$ converges weakly (as $n\to \infty$) to the Rayleigh distribution with density $\phi$. Hence
		\begin{align*}
		& \lim_{n\to \infty} \mathbb{E}_{\frac{-x}{\sqrt{b_n}}}\left(\phi\left(\frac{\hat{      B}_n(k_n-p)+\frac{x}{\sqrt{b_n}}}{\sqrt{k_n-p}}\right)\big\vert \hat{B}_n(k)\geq 0, k\leq k_n-p\right)=\int_{0}^{\infty} \phi(t)^2 dt=  \frac{\sqrt{\pi}}{4}
		\end{align*}
		uniformly in $x\in[-r_n,0]$ .
		Combining this with \eqref{medd} we conclude that \eqref{function estimate} holds, which allows us to complete the proof by successive approximations.
		\end{proof}
\subsubsection{KMT coupling for random walk} 
 We now introduce the well-known KMT Theorem  \cite{zbMATH03481611} which is an approximation
 method of a random walk satisfying $(\mathbf{H2})$ by a Gaussian random walk. It allows us
 to link estimates on random walks satisfying $(\mathbf{H2})$ to the ones previously
 proved under assumption $(\mathbf{H1})$.

\begin{theorem}[Komlos-Major-Tusnàdy]
	\label{Komlos}
	Let $(X_i)_{1\leq i\leq n}$ be a sequence of i.i.d random variables such that $\mathbb{E}(X_i))=0$, $0<\mathbb{E}(X_i^2)=\sigma^2<\infty$ and $\mathbb{E}(\exp(\theta|X_i|))<\infty$ for some $\theta>0$. Then there exists a  sequence of i.i.d standard normal variables $(Z_i)_{1\leq i\leq n}$ such that for all $y\geq 0$
	\begin{align*}
	&\mathbb{P}\left( \max_{1\leq k\leq n}|\sigma^{-1}\sum_{i=1}^{k}X_i-\sum_{i=1}^{k}Z_i|>C\log(n)+y\right)\leq K \exp{(-\lambda y)}
	\end{align*} 
	where $C, K$ and $\lambda$ are universal positive constants.
	\end{theorem}
Observe that we can construct an increasing sequence of integers $d_n$ such that
		$d_n/\log n \to \infty$ and $d_n^2/b_n\to 0$. Using assumption \eqref{equation 6}, we can apply Theorem \ref{Komlos}, to note
		that with high probability,  the random walk $(\bar{T}_k)_{k\leq k_n}$   stays within distance $d_n$ from  a
		random walk $(\hat{S}_k)_{0\leq k\leq k_n}$ with Gaussian steps. More precisely, setting the event
		$\mathcal{W}_n=\left\{|\bar{T_k}-\sqrt{b_n}\hat{S}_k|\leq d_n,k\leq k_n \right\}$   we have
		$$\mathbb{P}(\mathcal{W}_n^c)\leq e^{-\lambda(d_n - C \log n)} = o(n^{-\gamma})$$
		for all $\gamma > 0$, as $d_n \gg \log n$.
		We start by proving a version of Lemma \ref{MA} for the random walk satisfying $(\mathbf{H_2})$.
	\begin{lemma}
	\label{MAbis}
	Fix $\epsilon>0$, there exists $C>0$ such that  $$ \mathbb{P}(\bar{T}_k \geq -(k^{1/2-\epsilon}+y),k\leq k_n)\leq C\frac{1+\frac{y}{\sqrt{b_n}}}{\sqrt{k_n}}$$ for any $y>0$.
\end{lemma}
\begin{proof}
        Let $\epsilon>0$. The proof is an application  Theorem \ref{Komlos}. We have
		\begin{align*}
	&\mathbb{P}\left(\bar{T}_k \geq -(k^{1/2-\epsilon}-y),k\leq k_n\right)\\
	&\leq \mathbb{P}\left(\bar{T}_k \geq -(k^{1/2-\epsilon}+y),k\leq k_n,\mathcal{W}_n^c,k\leq k_n\right)+\mathbb{P}\left(\sqrt{b_n}\hat{S}_{k}\geq -(k^{1/2-\epsilon}+y)+d_n), k\leq k_n\right).
	\end{align*}
	 Applying Theorem \ref{Komlos}, there exists a constant $C>0$ such that
	\begin{align*}
		&\mathbb{P}\left(\bar{T}_k \geq -(k^{1/2-\epsilon}+y),k\leq k_n\right)\\
	&\leq  \frac{C}{n^{\gamma}}+\mathbb{P}\left(\sqrt{b_n}\hat{S}_{k}\geq -(k^{1/2-\epsilon}+y)+d_n, k\leq k_n\right),
	\end{align*}
 for all $\gamma>0$. Using the fact that $\frac{d_n}{\sqrt{b_n}}\to 0$ as $n\to \infty$, we have for $n$ large enough
 \begin{align*}
     &\mathbb{P}\left(\bar{T}_k \geq -(k^{1/2-\epsilon}+y),k\leq k_n\right)\leq \mathbb{P}\left(\hat{S}_{k}\geq -\frac{1}{\sqrt{b_n}}(k^{1/2-\epsilon}+y)+1  ,k\leq k_n\right)
 \end{align*}
 and by Lemma \ref{MA} we conclude the proof.
\end{proof}

	In the same way we prove a similar result to Lemma \ref{10}.
\begin{lemma}
		\label{7bis}
		Let $\alpha>0$, and for $0\leq k\leq k_n$, write $f_n(k)=\alpha \log \left(\frac{(k_n-k)b_n+1}{k_nb_n}\right)$. There exists $C>0$ such that for all $x \geq 0$,   $a \leq b\in{\mathbb{R}}$, we have for $n$ large enough
		\begin{align*}
		&\mathbb{P}\left(\bar{T}_{k_n}-f_n(k_n)\in{[a,b]}, \bar{T_k}\leq f_n(k)+x, k\leq k_n\right)\leq C(b-a+2d_n)\frac{(1+\frac{x}{\sqrt{b_n}})^{2}}{\sqrt{b_n\sigma^{2}}k_n^{\frac{3}{2}}}.
		\end{align*}
		\end{lemma}
		\begin{proof} 
	Applying  again Theorem \ref{Komlos}, there exists a constant $C>0$ such that
	\begin{align*}
	&\mathbb{P}\left(\bar{T}_{k_n}-f_n(k_n)\in[a,b], \bar{T_k}\leq f_n(k)+x, k\leq k_n\right)\\&\label{p}\leq \frac{C}{n^{\gamma}}+\mathbb{P}\left(\sqrt{b_n}\hat{S}_{k_n}-f_n(k_n)\in[a-d_n,b+d_n], \hat{S}_{k}\leq\frac{1}{\sqrt{b_n}} (f_n(k)+x+d_n),k\leq k_n\right).
	\end{align*}
for all $\gamma>0$.
	 For $n$ large enough we obtain
	 \begin{align*}
	 	&\mathbb{P}\left(\bar{T}_{k_n}-f_n(k_n)\in[a,b], \bar{T_k}\leq f_n(k)+x, k\leq k_n\right)\\&\leq \mathbb{P}\left(\sqrt{b_n}\hat{S}_{k_n}-f_n(k_n)\in[a-d_n,b+d_n],\hat{S}_{k}\leq\frac{1}{\sqrt{b_n}}(f_n(k)+x)+2,k\leq k_n\right),
		\end{align*}
		and we use Lemma \ref{10} to complete the proof.
		\end{proof}
	We now prove  Lemma \ref{Stone barrier} for  random walk satisfying $(\mathbf{H_2})$. 
	Set $$\bar{F}_n(k)=\frac{k}{k_n}a_n-c_n\mathbf{1}_{k\ne0,k_n}$$ $k=0,..k_n$ 	where $(c_n)_{n\in{\mathbb{N}}}$ is a sequence of integers satisfying $\lim_{n\to\infty}c_n=\infty$ and
	\begin{align}
	    &\lim_{n\to\infty}\frac{c_n}{\sqrt{b_n}}=0.
	\end{align} 
	Before moving to the proof we need the following Lemma.
	\begin{lemma}
	\label{323}
	    Uniformly in $x\in{[c_n,\infty]}$, we have 
	    \begin{align*}
	       & \lim_{n\to \infty} \left| \frac{L(\frac{x-c_n}{\sqrt{b_n}})}{L(\frac{x}{\sqrt{b_n}})}-1\right|=0.
	    \end{align*}
	\end{lemma}
	\begin{proof}
	     The proof follows from the properties of the renewal function introduced in \eqref{mhhh}. We first consider the ratio $\frac{L(y-\frac{c_n}{\sqrt{b_n}})}{L(y)}$ for large value of $y$. Let $\epsilon>0$, by \eqref{565}, there exists a constant $A=A(\epsilon)>0$ sufficiently large such that for all $y\geq A$ we have 
	     \begin{align}
	     \label{equation 27}
	         &c_0(1-\epsilon)y\leq L(y)\leq c_0(1+\epsilon)y.
	     \end{align}
	    Using  \eqref{equation 27}  we have
	     \begin{align}
	    \label{Ren1}
	        &\sup_{\frac{x}{\sqrt{b_n}}\geq A}\left| \frac{L(\frac{x-c_n}{\sqrt{b_n}})}{L(\frac{x}{\sqrt{b_n}})}-1\right|\leq |\frac{1+\epsilon}{1-\epsilon}-1|+\frac{c_0(1+\epsilon)c_n}{A\sqrt{b_n}}.
	    \end{align}
	   
	    On the other hand, recall that $y\mapsto L(y)$  is continuous. Hence it is uniformly continuous on
	    $[0,A]$, and
	    
	    \begin{align*}
	        & \sup_{x\in[c_n,A\sqrt{b_n}]}\left| L(\frac{x}{\sqrt{b_n}})-L(\frac{x-c_n}{\sqrt{b_n}})\right|\\&\leq\sup_{y\in[\frac{c_n}{\sqrt{b_n}},A]}\left| L(y)-L(y-\frac{c_n}{\sqrt{b_n}})\right|\leq w_{L}(\frac{c_n}{\sqrt{b_n}}),
	    \end{align*}
	    where $w_L(\delta)=\displaystyle\sup_{\substack{s,t \\|t-s|\leq \delta}}|L(t)-L(s)|$. Since  the renewal function is increasing, we have $L(y)\geq 1$, for all $y\geq 0$, which implies that 
	    \begin{align}
	    \label{Ren2}
	    &\sup_{y\in[\frac{c_n}{\sqrt{b_n}},A]}\left| \frac{L(y-\frac{c_n}{\sqrt{b_n}})}{L(y)}-1\right|\leq   w_{L}(\frac{c_n}{\sqrt{b_n}}) \to 0 \text{ as } n \to \infty.
	    \end{align}
	     By \eqref{Ren1} and \eqref{Ren2}  we have 
	       \begin{eqnarray*}
	         \sup_{x\in{[c_n,\infty)}}\left| \frac{L(\frac{x-c_n}{\sqrt{b_n}})}{L(\frac{x}{\sqrt{b_n}})}-1\right|&\leq & \sup_{x\in[c_n,A\sqrt{b_n}]}\left| \frac{L(\frac{x-c_n}{\sqrt{b_n}})}{L(\frac{x}{\sqrt{b_n}})}-1\right|+ \sup_{\frac{x}{\sqrt{b_n}}\geq A}\left| \frac{L(\frac{x-c_n}{\sqrt{b_n}})}{L(\frac{x}{\sqrt{b_n}})}-1\right|\\&\leq &  w_{L}(\frac{c_n}{\sqrt{b_n}})+|\frac{1+\epsilon}{1-\epsilon}-1|+\frac{c_0(1+\epsilon)c_n}{A\sqrt{b_n}}.
	    \end{eqnarray*}
	    
	     Letting $n\to \infty$ then $\epsilon \to 0$ we conclude the proof.
	    \end{proof}
	\begin{corollary}
	\label{stonebis}
		 Setting $$a_n=\frac{-3}{2\theta^{*}}\ln(n)+\frac{\log(b_n)}{\theta^{*}},$$ then there exists $C>0$ such that 
		With the same notation as Lemma \ref{Stone barrier} we have  
		\begin{align}
		\label{sb}
		&\mathbb{E}\left(f(\bar{T}_{k_n}-a_n+x)e^{-\theta^{*}\bar{T}_{k_n}}\mathbf{1}_{\{\bar{T}_{k}\leq \bar{F}_n(k)-x,k\leq k_n\}}\right)=\frac{e^{\theta^{*}x}}{\sqrt{2\pi\sigma^2 b_n}}\int_{-\infty}^{0} f(y)e^{-\theta^{*}y}dy\left(L(\frac{-x}{\sqrt{b_n}})+o(1)\right),
		\end{align}
		uniformly in $x \in{[-r_n, -(c_n+d_n)]}$.
	\end{corollary}
	\begin{proof}

	We will divide the proof of this corollary into two parts, proving separately an upper and a lower bound for \eqref{sb}.
		We set $h(z)=e^{-\theta^* z}f(z)$ . As a consequence, it is enough to prove that 
		\begin{align*}
		\label{30}
		&\mathbb{E}\left(h(\bar{T}_{k_n}-a_n+x)\mathbf{1}_{\{\bar{T}_{k}\leq\bar{F}_n(k)-x,k\leq k_n\}}\right)=
		\frac{1}{k_n^{3/2}\sqrt{2
		\pi \sigma^2 b_n}}\int_{-\infty}^{0} h(y)dy(	L(\frac{-x}{\sqrt{b_n}})+o(1))
		\end{align*}
		uniformly in $x\in{[-r_n,-(c_n+d_n)]}$.  Using the same arguments in Lemma \ref{Stone barrier}, it is enough to prove this Corollary with the function $h(z) =\mathbf{1}_{[-a,0]}(z)$ for some $a>0$. Then we write  
		\begin{align}
		&\mathbb{E}\left(h(\bar{T}_{k_n}-a_n+x)\mathbf{1}_{\{\bar{T}_{k}\leq \bar{F}_n(k)-x,k\leq k_n\}}\right)=\mathbb{P}(\mathcal{A}^{(k)}_n(x))
		\end{align}
	where  $$\mathcal{A}^{(k)}_n(x)=\left\{\bar{T}_{k}\leq \bar{F}_n(k)-x,k\leq k_n,\bar{T}_{k_n}-a_n+x\geq -a \right\}$$  the event that the random walk $(\bar{T}_{k})_{k\leq k_n}$  stays below the barrier $k\mapsto F_n(k)$ for all $k\leq k_n$  and end up in a finite interval.  
	\begin{itemize}
	    \item \textbf{Upper bound}
	\end{itemize}
		We have
		\begin{align}
		&\label{quation12}\mathbb{P}(\mathcal{A}^{(k)}_n(x))= \mathbb{P}\left(\mathcal{A}^{(k)}_n(x),|\bar{T}_{k_n-1}-a_n+x|\leq h_n\sqrt{b_n}\right)+\mathbb{P}\left(\mathcal{A}^{(k)}_n(x),|\bar{T}_{k_n-1}-a_n+x|> h_n\sqrt{b_n}\right)		
		\end{align}
		where $(h_n)_{n\in{\mathbb{N}}}$ is a sequence growing to $\infty$, that we will fix later on. We bound these two quantities separately choosing $h_n$ such that
		\begin{align*}
		   &\limsup_{n\to \infty}\sup_{x\in{[-r_n,-c_n]}}L(\frac{-x}{\sqrt{b_n}})^{-1}k_n^{3/2}\sqrt{b_n}\mathbb{P}\left(\mathcal{A}^{(k)}_n(x),|\bar{T}_{k_n-1}-a_n+x|>h_n\sqrt{b_n}\right) =0.
		\end{align*}
			We first  observe that
	\begin{align*}
	   & \mathbb{P}\left(\mathcal{A}^{(k)}_n(x),|\bar{T}_{k_n-1}-a_n+x|>h_n\sqrt{b_n}\right)
	   =\mathbb{P}\left(\mathcal{A}^{(k)}_n(x),\bar{T}_{k_n-1}-a_n+x<-h_n\sqrt{b_n}\right)
	\end{align*}
	since the event $\left\{\bar{T}_{k_n-1}\leq \bar{F}_n(k_n-1)-x,\bar{T}_{k_n-1}-a_n+x>h_n\sqrt{b_n}\right\}$ is impossible for $n$ large enough.
		Then we have,
	\begin{align*}
	      &\limsup_{n\to \infty}\sup_{x\in{[-r_n,-c_n]}}L(\frac{-x}{\sqrt{b_n}})^{-1}k_n^{3/2}\sqrt{b_n}\mathbb{P}\left(\mathcal{A}^{(k)}_n(x),\bar{T}_{k_n-1}-a_n+x<-h_n\sqrt{b_n}\right)\\&\leq   \limsup_{n\to \infty}\sup_{x\in{[-r_n,-c_n]}}L(\frac{-x}{\sqrt{b_n}})^{-1}k_n^{3/2}\sqrt{b_n}\mathbb{P}\left(\bar{T}_{k_n}-a_n+x\in[-a,0],\bar{T}_{k_n-1}-a_n+x<-h_n\sqrt{b_n}\right)\\&\leq\limsup_{n\to \infty}\sup_{x\in{[-r_n,-c_n]}}L(\frac{-x}{\sqrt{b_n}})^{-1}k_n^{3/2}\sqrt{b_n}\mathbb{P}(|\bar{T}_1|\geq \sqrt{b_n} h_n).
	\end{align*}
	On the other hand we write
\begin{align*}
 & \mathbb{P}(|\bar{T}_1|\geq \sqrt{b_n} h_n) = \mathbb{P}(\bar{T}_1\geq \sqrt{b_n} h_n) + \mathbb{P}(- \bar{T}_1\geq \sqrt{b_n}h_n),
  \end{align*}
and additionally we have
\begin{align*}
  &\mathbb{P}\left(\bar{T}_1 \geq \sqrt{b_n} h_n\right)=\mathbb{P}\left(e^{\theta \bar{T}_1}\geq e^{\theta h_n\sqrt{b_n}}\right) \leq e^{b_n(\Lambda(\theta)-\theta h_n \sqrt{b_n})},
\end{align*}
for all $\theta>0.$
As $\Lambda$ is $\mathcal{C}^2$ on a neighbourhood of $0$, we have
\begin{align*}
   & \Lambda(\theta) = \Lambda(0) + \Lambda'(0)\theta + \frac{\Lambda''(0)\theta^2}{2} + o(\theta^2)) = \Lambda''(0)\theta^2/2 + o(\theta^2).
\end{align*}
 
Therefore, choosing $\theta=\frac{h_n}{\sqrt{b_n}}$, there exists $C>0$ such that for all $n$ large enough, we have
\begin{align*}
    &\mathbb{P}(\bar{T}_1\geq \sqrt{b_n} h_n) \leq  e^{-C h^2_n}.
\end{align*}
  
Similarly, we have $\mathbb{P}(\bar{T}_1\leq -\sqrt{b_n} h_n) \leq  e^{-C h^2_n}$ for $n$ large enough. Finally we obtain
\begin{align*}
	      &\limsup_{n\to \infty}\sup_{x\in{[-r_n,-c_n]}}L(\frac{-x}{\sqrt{b_n}})^{-1}k_n^{3/2}\sqrt{b_n}\mathbb{P}\left(\mathcal{A}^{(k)}_n(x),\bar{T}_{k_n-1}-a_n+x<-h_n\sqrt{b_n}\right)\\&\leq\limsup_{n\to \infty}\sup_{x\in{[-r_n,-c_n]}}L(\frac{-x}{\sqrt{b_n}})^{-1}k_n^{3/2}\sqrt{b_n}e^{-Ch^2_n}
	\end{align*}
	which goes to zero as $n \to \infty$ as long as  $h_n> 2\sqrt{\frac{\log(n)}{C}}$.
	
	 We now bound the first quantity in the right hand-side of \eqref{quation12}.  
		Applying the Markov property at time $k_n-1$ we get
		\begin{align}
		\label{112}
		&\mathbb{P}\left(\mathcal{A}^{(k)}_n(x),|\bar{T}_{k_n-1}-a_n+x|\leq h_n\sqrt{b_n}\right)\\&\nonumber=\mathbb{E}\left(f_{n}(\bar{T}_{j}+a_n-x)\mathbf{1}_{\left\{\bar{T_k}\leq\bar{F}_n(k)-x,|\bar{T}_{k_n-1}-a_n+x|\leq h_n\sqrt{b_n}, k\leq k_n-1|\right\}}\right)
		\end{align}
		 where $f_{n}(z)=\mathbb{P}_z\left(\bar{T}_{1}\in{[-a,0]}\right)$ for all  $z\in{\mathbb{R}}$. We estimate the function  $z\mapsto f_{n}(z)$, $n\in{\mathbb{N}}$, using the refined Stone's local limit theorem in \cite[Theorem 2.1]{borovkov2017generalization}.
		By $(\mathbf{H_2})$ there exists a constant $c>0$ such that for all $z\in{\mathbb{R}}$ 
			\begin{align*}
		&f_{n}(z)\leq \frac{a}{\sqrt{2\pi b_n\sigma^2}}\exp\left(\frac{-(z-a_n+x)^2}{2 b_n}\right)+\frac{c}{b_n}.
	\end{align*} 
			To approximate \eqref{112} for a random walk satisfying $(\mathbf{H}_2)$
			we  apply Theorem \ref{Komlos}, there exists a constant $C>0$ such that for all $\gamma>0$ we have 
		\begin{align}
		&\label{mriguel}
	\mathbb{P}\left(\mathcal{A}^{(k)}_n(x),|\bar{T}_{k_n-1}-a_n+x|\leq h_n\sqrt{b_n}\right)	\\&\nonumber\leq\frac{C}{n^{\gamma}}+\frac{a}{\sqrt{2\pi b_n\sigma^2}} \mathbb{E}\left(\sup_{|w|\leq d_n}\exp\left(\frac{-(\hat{S}_{k_n-1}+w-a_n+x)^2}{2 b_n\sigma^2}\right)\mathbf{1}_{\mathcal{B}^{(k)}_n(x)}\right)+\frac{c}{b_n}\mathbb{P}\left(\mathcal{B}^{(k)}_n(x)\right).
	\end{align} 
  where $\mathcal{B}^{(k)}_n(x)=\left\{\hat{S}_{k}\leq \tilde{F}_n(k)+d_n-x,k\leq k_n-1,|\hat{S}_{k_n-1}-a_n+x|\leq h_n \sqrt{b_n}+d_n\right\}$. We write then 
		\begin{align}
		\label{222}
		    	&\mathbb{E}\left(\sup_{|w|\leq d_n}\exp\left(\frac{-(\hat{S}_{k_n-1}+x-a_n+w)^2}{2 b_n}\right)\mathbf{1}_{\mathcal{B}^{(k)}_n(x)}\right)\\ &\nonumber=\mathbb{E}_{\mathbb{Q}}\left(\sup_{|w|\leq d_n}e^{-\frac{a_n}{n}\sqrt{b_n}\tilde{S}_{k_n-1}+n\Lambda(a_n/n)}\exp\left(\frac{-(\sqrt{b_n}\tilde{S}_{k_n-1}+x-a_n+w)^2}{2 b_n}\right)\mathbf{1}_{\mathcal{C}^{k}_n(x)}\right) 
		\end{align}
		 where $$\mathcal{C}^{(k)}_n(x)=\left\{\tilde{S}_{k}\leq \frac{-x+d_n-c_n}{\sqrt{b_n}},k\leq k_n-1,|\sqrt{b_n}\tilde{S}_{k_n-1} +\frac{a_n}{k_n}-x|\le h_n\sqrt{b_n}+d_n\right\}$$ and
		 $\sqrt{b_n}\tilde{S}_k=\sqrt{b_n}\hat{S}_k-\frac{k}{k_n}a_n$ is a centred Gaussian random walk under the measure $\mathbb{Q}$ defined in the proof of Lemma \ref{Stone barrier}.
		Then using the dominated convergence theorem  and the fact that $\frac{d_n^2}{b_n} \to_{n\to\infty}0$, and rescaling by $\frac{1}{\sqrt{b_n}}$ we obtain
	\begin{align*}
		&\limsup_{n\to \infty}\mathbb{E}\left(\exp\left(\sup_{|w|\leq d_n}\frac{-(\hat{S}_{k_n-1}+w -a_n+x)^2}{2 b_n\sigma^2}\right)\mathbf{1}_{\mathcal{C}^{(k)}_n(x)}\right)\\&\leq \limsup_{n\to \infty}\mathbb{E}_{\mathbb{Q}}\left(\exp\left(\frac{-(\sqrt{b_n}\tilde{S}_{k_n-1}+x)^2}{2 b_n\sigma^2}\right)\mathbf{1}_{\left\{\tilde{S}_{k}\leq \frac{1}{\sqrt{b_n}}(x+d_n-c_n),k\leq k_n-1\right\}}\right)\\& \leq  \limsup_{n\to \infty}\mathbb{E}_{\mathbb{Q}}\left(\exp\left(\frac{-(\tilde{S}_{k_n-1}+\frac{-x}{\sqrt{b_n}})^2}{2\sigma^2}\right)\mathbf{1}_{\left\{\tilde{S}_{k} \leq \frac{-x+d_n-c_n}{\sqrt{b_n}}, k\leq k_n-1\right\}}\right).
		\end{align*}
		Applying Lemma $2.3$ in \cite{zbMATH06216112} for the  Gaussian random walk $(\tilde{S}_k)_{k\leq k_n}$, and using Lemma \ref{323}  we obtain for all $\frac{x}{\sqrt{b_n}}\in{[-\frac{r_n}{\sqrt{b_n}},-\frac{(c_n+d_n)}{\sqrt{b_n}}]}$, 
		\begin{align*}
		   &\limsup_{n\to \infty}k_n^{\frac{3}{2}} \mathbb{E}_{\mathbb{Q}}\left(\exp\left(\frac{-(\tilde{S}_{k_n-1}+\frac{-x}{\sqrt{b_n}})^2}{2\sigma^2}\right)\mathbf{1}_{\left\{\tilde{S}_{k} \leq \frac{-x+d_n-c_n}{\sqrt{b_n}}, k\leq k_n-1\right\}}\right)
		   \\&\leq  \frac{1}{\sqrt{2\pi \sigma^2}}\int_{0}^{\infty} e^{-\frac{y^2}{2\sigma^2}}L(y)dy\times L(\frac{-x+d_n-c_n}{\sqrt{b_n}}).
		\end{align*}
		We now observe that by \eqref{mhhh},  $\frac{1}{\sqrt{2\pi\sigma^2}}\int_{0}^{\infty}L(y)e^{-\frac{y^2}{2\sigma^2}} dy=\mathbb{E}\left(L(\hat{S}_1)\mathbf{1}_{\left\{\hat{S}_1\geq 0\right\}}\right))=L(0)=1$, which implies 
		\begin{align*}
		     &\limsup_{n\to \infty} k_n^{\frac{3}{2}}\mathbb{E}_{\mathbb{Q}}\left(\exp\left(\frac{-(\sqrt{b_n}\tilde{S}_{k_n-1}+x)^2}{2 b_n\sigma^2}\right)\mathbf{1}_{\left\{\tilde{S}_{k}\leq \frac{-x+d_n-c_n}{\sqrt{b_n}}, k\leq k_n-1\right\}}\right)\\& \leq  L(\frac{-x+d_n-c_n}{\sqrt{b_n}}).
		\end{align*}
	
			To complete the proof of the upper bound it remains to show that  	
		\begin{align}
		\label{300}
		  &\limsup_{n\to \infty}\sup_{x\in{[-r_n,-c_n]}} L(\frac{-x}{\sqrt{b_n}})^{-1}k_n^{3/2}b_n^{-\frac{1}{2}}\mathbb{E}\left(\mathbf{1}_{\mathcal{B}^{k}_n(x)}\right)=0.
		\end{align}
	 Using similar computations  we have  
	\begin{align*}
	     &\sup_{x\in{[-r_n,-c_n]}}k_n^{\frac{3}{2}}
	     \mathbb{P}\left(\hat{S}_{k}\leq \tilde{F}_n(k)+d_n-x,k\leq k_n-1,|\hat{S}_{k_n-1}-a_n+x|\leq h_n \sqrt{b_n}+d_n\right)\\&\leq \sup_{x\in{[-r_n,-c_n]}}k_n^{\frac{3}{2}} \mathbb{P}\left(\tilde{S}_{k}\leq \frac{-x+d_n-c_n}{\sqrt{b_n}},\tilde{S}_{k_n-1}-\frac{a_n}{k_n\sqrt{b_n}}+\frac{x}{\sqrt{b_n}}\geq -h_n-\frac{d_n}{\sqrt{b_n}}, k\leq k_n-1\right)\\&\leq C L(\frac{-x+d_n-c_n}{\sqrt{b_n}})\int_{0}^{h_n+\frac{2d_n}{\sqrt{b_n}}}L(y)dy.
	     \end{align*}
	     Since the renewal function $x\mapsto R(x)$ is increasing we have by \eqref{1v1}
	     \begin{align*}
	         &\sup_{x\in{[-r_n,-c_n]}} 
	        k_n^{\frac{3}{2}} \mathbb{P}\left(\hat{S}_{k}\leq \tilde{F}_n(k)+d_n-x,k\leq k_n-1,|\hat{S}_{k_n-1}-a_n+x|\leq h_n \sqrt{b_n}+d_n\right)
	     \\& \leq  C (h_n+2\frac{d_n}{\sqrt{b_n}})L(\frac{-x+d_n-c_n}{\sqrt{b_n}}) L(h_n+\frac{d_n}{\sqrt{b_n}}) \leq C (h_n+\frac{2d_n}{\sqrt{b_n}})^2 L(\frac{-x+d_n-c_n}{\sqrt{b_n}})
	     \end{align*}
	      where we used in the last inequality that $L(t_n)\leq c t_n$ for some constant $c>0$ when $t_n \to_{n\to \infty} \infty$.
	      Thanks to Lemma \ref{323}  we obtain 
	      \begin{align*}
	           &\sup_{x\in{[-r_n,-c_n]}}L(\frac{-x}{\sqrt{b_n}})^{-1}k_n^{3/2}b_n^{-\frac{1}{2}}\mathbb{P}\left(\mathcal{B}^{(k)}_n(x)\right)\leq  C \frac{(h_n+\frac{2d_n}{\sqrt{b_n}})^2}{\sqrt{b_n}} 
	      \end{align*}
	       which goes to zero since $\frac{\log(n)+d_n}{\sqrt{b_n}}\to 0$ as $n\to \infty$.
	       Finally by \eqref{mriguel}, \eqref{300} and Lemma \ref{323} we deduce that
	       \begin{align}
		&\label{upper}
		  \limsup_{n\to \infty}\sup_{x\in{[-r_n,-c_n]}}L(\frac{-x}{\sqrt{b_n}})^{-1}\sqrt{b_n}k_n^{\frac{3}{2}}\mathbb{E}(h(\bar{T}_{k_n}-a_n+x)\mathbf{1}_{\{\bar{T}_{k}\leq \bar{F}_n(k)-x,k\leq k_n\}})	\\& \nonumber \leq \limsup_{n\to \infty}\sup_{x\in{[-r_n,-c_n]}}L(\frac{-x}{\sqrt{b_n}})^{-1}\sqrt{b_n}k_n^{\frac{3}{2}} \mathbb{P}\left(\mathcal{A}^{(k)}_n(x),|\bar{T}_{k_n-1}-a_n+x|\leq h_n\sqrt{b_n}\right)\\&\nonumber\leq \frac{1}{\sqrt{2\sigma^2 \pi}}\int_{-\infty}^{0} f(y)e^{-\theta^{*}y}dy.
		\end{align}
		 We now treat the lower bound. 
		 	\begin{itemize}
	    \item \textbf{Lower bound}
	\end{itemize}
We now compute a lower bound for $\mathbb{P}\left(\mathcal{A}^{(k)}_n(x)\right)$.		 
		Using similar arguments to these used in the upper bound 
		we have by Theorem \ref{Komlos}  
		\begin{align*}
		&\label{mriguel}
	\mathbb{P}\left(\bar{T}_{k}\leq \bar{F}_n(k)-x,k\leq k_n,\bar{T}_{k_n}-a_n+x\in[-a,0]\right)	\\&\geq \frac{a}{\sqrt{2\pi b_n\sigma^2}}\mathbb{E}\left(\inf_{|w|\leq d_n}\exp\left(\frac{-(  \sqrt{b_n}\hat{S}_{k_n-1}+w-a_n+x)^2}{2 b_n\sigma^2}\right)\mathbf{1}_{ \left\{\sqrt{b_n}\tilde{S}_k\leq \bar{F}_n(k) -d_n-x, k\leq k_n-1\right\}}\right).
	\end{align*} 
	  Similar computations to these used in the upper bound lead to
		\begin{align*}
		    	&\liminf_{n\to \infty}\mathbb{E}\left(\inf_{|w|\leq d_n}\exp\left(\frac{-(\sqrt{b_n}\hat{S}_{k_n-1}+x-\frac{a_n}{n}+w)^2}{2 b_n}\right)\mathbf{1}_{\left\{\hat{S}_{k}\leq \bar{F}_n(k)-x-d_n,k\leq k_n-1\right\}}\right)\\&\geq \liminf_{n\to \infty}\mathbb{E}_{\mathbb{Q}}\left(\exp\left(\frac{-(\tilde{S}_{k_n-1}+\frac{x}{\sqrt{b_n}})^2}{2\sigma^2}\right)\mathbf{1}_{\left\{\tilde{S}_{k}\leq \frac{-x-d_n-c_n}{\sqrt{b_n}}, k\leq k_n-1\right\}}\right).
		    	\end{align*}
		Applying again Lemma $2.3$ in \cite{zbMATH06216112} for the  Gaussian random walk $(\tilde{S}_k)_{k\leq k_n}$, we obtain for all $x\in{[-r_n,-(c_n+d_n)]}$,
		\begin{align*}
		   &\liminf_{n\to \infty}k_n^{\frac{3}{2}}\mathbb{E}_{\mathbb{Q}}\left(\exp\left(\frac{-(\tilde{S}_{k_n-1}+\frac{x}{\sqrt{b_n}})^2}{2\sigma^2}\right)\mathbf{1}_{\left\{\tilde{S}_{k}\leq \frac{-x-d_n-c_n}{\sqrt{b_n}}, k\leq k_n-1\right\}}\right)\\&\geq  \frac{1}{\sqrt{2\pi\sigma^2}}\int_{0}^{\infty} e^{-\frac{y^2}{2\sigma^2}}R(y)dy\times L(\frac{-x-(c_n+d_n)}{\sqrt{b_n}}).
		\end{align*}
		Finally, by Lemma \ref{323} we get  
		\begin{align}
		 & \liminf_{n\to \infty}\inf_{x\in{[-r_n,-(c_n+d_n)]}}(L(\frac{-x}{\sqrt{b_n}})^{-1}\sqrt{b_n}k_n^{\frac{3}{2}}\mathbb{E}(h(\bar{T}_{k_n}-a_n+x)\mathbf{1}_{\{\bar{T}_{k}\leq \bar{F}_n(k)-x,k\leq k_n\}})\\& \geq \nonumber \frac{1}{\sqrt{2\pi\sigma^2}}\int_{-\infty}^{0} f(y)e^{-\theta^{*}y}dy.
		\end{align}
		Combining equations \eqref{upper} and \eqref{aminov} we deduce that
		\begin{align*}
		&\mathbb{E}\left(h(\bar{T}_{k_n}-a_n+x)\mathbf{1}_{\{\bar{T}_{k}\leq \bar{F}_n(k)-x,k\leq k_n\}}\right)=\frac{e^{\theta^*x}}{\sqrt{2\pi\sigma^2}\sqrt{b_n}k_n^{\frac{3}{2}}}\int_{-\infty}^{0} f(y)e^{-\theta^{*}y}dy(L(\frac{-x}{\sqrt{b_n}})+o(1))
		\end{align*}
		uniformly in $x \in{[-r_n, -(c_n+d_n)]}$.
	\end{proof}
	\section{The modified extremal process}
	Recall that $(\bar{T_k})_{1\leq k\leq  k_n}$ is a sequence of centred random walk with $\mathrm{Var}(\bar{T_k})=kb_n \sigma^2$ and $a_n=m_n-k_nb_nv$. The goal of the next section is to introduce a modified extremal process and to prove that it has the same weak limit as the original extremal process $\mathcal{E}_n$.
	
	 Start by setting  $$\mathcal{E}_{n,R_n}=\sum_{u\in{\mathcal{H}_{k_n}}}\delta_{S_u-m_n}\mathbf{1}_{\{S_{u_{k}}\leq R_n(k), \forall  k\leq k_n\}},$$
	where refer to the function $R_n:\{0...,k_n\}\mapsto \mathbb{R}$ as a barrier. 
	More precisely	our objective is to prove that the weak limit of the modified extremal process $\mathcal{E}_{n,R_n}$ and the original extremal process $\mathcal{E}_{n}$ coincide for a well-chosen function $R_n$.	The main steps of the proof of Theorem \ref{Theorem} are the following:
	\begin{itemize}
		\item We show that there exists an upper barrier such that, with high probability, all individuals stay below it all most of time.
		\item  The second step  is to locate the paths of extremal individuals. Here the method is inspired from the work of Arguin, Bovier and Kistler in \cite{zbMATH05975783} in the context of branching Brownian motion. 
		\item We show that all individuals contributing in the extremal process split from the root. 
	\end{itemize}
 We start by proving that there exists a barrier $R_n$ such that, with high probability, all individuals stay below it all most of time. 
	\begin{lemma}
	\label{lemma 14}
		Consider the barrier   	$$R_n(k)=kb_nv-\frac{3}{2\theta^{*}}\log(\frac{k_nb_n+1}{(k_n-k)b_n+1})+c_n,  k=0..., k_n$$ 
	where $(c_n)_{n\in{\mathbb{N}}}$ is the sequence of integers defined in \eqref{equation1551}.
		It then holds:
		\begin{align*}
		&
		\mathbb{P}(\exists u\in{\mathcal{H}_{k_n}}, S_{u_{k}}>R_n(k), \text{ for some }  k\leq k_n)=o(1) \text{ when } n \uparrow  \infty.
		\end{align*} 
	\end{lemma}
	\begin{proof}
		Using Markov inequality we get
		\begin{align*}
		&\mathbb{P}(\exists|u|=k_n, S_{u_{k}}>R_n(k),k\leq k_n)\leq\sum_{k\leq  k_n} \mathbb{E}\left(\sum_{|u|=k}\mathbf{1}_{\{ S_{u_k}>R_n(k),S_{u_{j}}\leq R_n(j),j<k\}}\right).
		\end{align*} 
		  By Proposition \ref{M} we have,
		\begin{align}
		&\nonumber\sum_{k\leq  k_n} \mathbb{E}\left(\sum_{|u|=k}\mathbf{1}_{\{ S_{u_k}>R_n(k),S_{u_{j}}\leq R_n(j),j<k\}}\right)\\&\nonumber\leq\sum_{k\leq k_n}\mathbb{E}\left(\exp(-\theta^{*}\bar{T}_{k})\mathbf{1}_{\{\bar{T}_{k}>R_n(k)-kb_nv,\bar{T}_{j}\leq R_n(j)-jb_nv,j<k \}}\right)\\ &\label{t}\leq e^{-\theta^{*}c_n}\sum_{k\leq k_n} \frac{(k_nb_n+1)^{\frac{3}{2}}}{(k_n-k)b_n+1)^{\frac{3}{2}}}\mathbb{P}(\bar{T}_{k}>\bar{R}_n(k),\bar{T}_{j}\leq \bar{R}_n(j),j< k ),
		\end{align}
		where $\bar{R}_n(j)= R_n(j)-jb_nv$, for all $1\leq j\leq k_n$. We compute this probability by conditioning  with respect to the last step $\bar{T}_k-\bar{T}_{k-1}$ to get 
		\begin{align*}
		\label{12}
		&\mathbb{P}(\bar{T}_{k}>\bar{R}_n(k),\bar{T}_{j}\leq \bar{R}_n(j),j\leq k )= \mathbb{E}(f_{k-1}(\bar{T}_k-\bar{T}_{k-1}))
		\end{align*}
		where, $\forall y\in{\mathbb{R}}$
		\begin{align*}
	&f_{k-1}(y)=\mathbb{P}\left(\bar{R}_n(k)-y\leq \bar{T}_{k-1}\leq \bar{R}_n(k) ,B_n(j)\leq\frac{1}{\sqrt{b_n}}\bar{R}_n(j)),j\leq k-1\right).
		\end{align*}
			Assume that $\mathbf{(H_1)}$ or $\mathbf{(H_2)}$ hold, by Lemma \ref{10} or \ref{7bis} and using the fact that $\frac{c_n}{\sqrt{b_n}}\to_{n\to\infty}0$, we deduce that, for $n$ large enough, there exists $C>0$ such that, $\forall y\in{\mathbb{R}}$
		\begin{align*} 
		&f_{k-1}(y)\leq C\mathbf{1}_{\left\{y\geq 0\right\}}\frac{(1+\frac{y}{\sqrt{b_n}})^{3}}{k^{3/2}}.
		\end{align*}
		 Now  Plugging this in \eqref{t} we obtain
		\begin{align*}
		&\mathbb{P}(\exists |u|=k_n, S_{u_k}>R_n(k),k\leq k_n)\\&\leq C e^{-\theta^{*}c_n}\sum_{k\leq k_n} \frac{(k_nb_n+1)^{\frac{3}{2}}}{((k_n-k)b_n+1)^{\frac{3}{2}}}\frac{1}{k^{\frac{3}{2}}}(1+ \mathbb{E}(\frac{\bar{T_k}-\bar{T}_{k-1}}{\sqrt{b_n}})^3_{+})\to_{n\to\infty}0.
		\end{align*}   
	 completing the proof.
\end{proof}
 From this lemma we deduce  that  the extremal process $\mathcal{E}_{n,R_n}$ has the same weak limit of the one of $\mathcal{E}_{n}$. 
 The second step of the proof of the main theorem is to locate the paths of extremal individuals. To do that we will consider a barrier which is lower.  For the choice of such  barrier we refer to the work of  Arguin, Bovier and Kistler \cite{zbMATH05975783} in the case of branching Brownian motion. 
	\begin{proposition}
		\label{pro11}
		Define the  barrier $$F_n(k)=kb_nv+\frac{k}{k_n}a_n-c_n\mathbf{1}_{k\ne 0,k_n},  k=0..., k_n.$$ 
		Let $A=[a,\infty)$ where $a\in{\mathbb{R}}$, then we have $$ \lim_{n\to \infty }\mathbb{E}(\mathcal{E}_{n,R_n}(A)-\mathcal{E}_{n,F_n}(A))=0.$$
	\end{proposition}
	\begin{proof}  Let the following subsets of $\mathcal{H}_{k_n}$ $$\mathbf{A}^{(u)}_{n}=\left\{u\in{\mathcal{H}_{k_n}}:  
		S_{k_n}-m_n \in{A}, S_k\leq R_n(k), k\leq k_n,\right\}$$  the set of particles at generation $k_n$ that are close to the maximum and that stay below the barrier $k\mapsto R_n(k)$ for all $k\leq k_n$. Respectively we introduce  $$\mathbf{B}^{(u)}_{n}=\left\{u\in{\mathcal{H}_{k_n}}:S_{k_n}-m_n \in{A}, S_k\leq F_n(k), k\leq k_n \right\}.$$ Set the integer-valued variable $$\#(\mathbf{A}^{(u)}_n \cap (\mathbf{B}^{(u)}_n)^{c})=\#\left\{ u\in{\mathcal{H}_{k_n}}: S_{k_n}-m_n \in{A}, S_k\leq R_n(k), k\leq k_n,\exists j\leq k_n,S_j>F_n(j) \right\}.$$ Using the fact that  $\mathbf{B}^{(u)}_n\subset \mathbf{A}^{(u)}_{n}$ we deduce that
		\begin{align*}
		&\mathbb{E}(\mathcal{E}_{n,R_n}(A)-\mathcal{E}_{n,F_n}(A))=\mathbb{E}(\#(\mathbf{A}^{(u)}_n \cap (\mathbf{B}^{(u)}_n)^c))\\&= \mathbb{E}\left(\sum_{|u|=k_n}\mathbf{1}_{\{S_{k_n}-m_n \in{A}, S_k\leq R_n(k), k\leq k_n,\exists j\leq k_n,S_j>F_n(j)\}}\right),
		\end{align*}
	then thanks to Proposition \ref{M} we obtain
		\begin{align*}
		&\mathbb{E}\left(\sum_{|u|=k_n}\mathbf{1}_{\{S_{k_n}-m_n \in{A}, S_k\leq R_n(k), k\leq k_n,\exists j\leq k_n,S_j>F_n(j)\}}\right)\\&=\mathbb{E}\left(e^{-\theta^{*}\bar{T}_{k_n}}\mathbf{1}_{\{\bar{T}_{k_n}-a_n \in{A}, \bar{T}_k\leq \bar{R}_n(k), k\leq k_n,\exists j\leq k_n,\bar{T}_j>\bar{F}_n(j)\}}\right),	
		\end{align*}	where $a_n=m_n-k_nb_nv$ and $\bar{F}_n(j)=\frac{j}{k_n} a_n-c_n\mathbf{1}_{j\ne 0,k_n}$ for all $1\leq j\leq k_n$.
		Summing with respect to the value of $T_{k_n}-a_n-a$ at time $k_n$, we have 
	\begin{align*}
	    &\mathbb{E}\left(e^{-\theta^{*}\bar{T}_{k_n}}\mathbf{1}_{\{\bar{T}_{k_n}-a_n \in{A}, \bar{T}_k\leq R_n(k)-kb_nv, k\leq k_n,\exists j\leq k_n,\bar{T}_j>\bar{F}_n(j)\}}\right)\\&\leq\frac{Cn^{\frac{3}{2}}}{{b_n}}\sum_{r\geq 0} e^{-\theta^*r}\sum_{0\leq j\leq k_n}\mathbb{P}\left(\bar{T}_{k_n}-a_n-a\in{[r,r+1]},\bar{T}_k\leq \bar{R}_n(k),k\leq k_n,\bar{T}_j>\bar{F}_n(j)\right).
	\end{align*}
		Applying the Markov property at time $j$ we get 
		\begin{align}
		&\nonumber\mathbb{P}\left(\bar{T}_{k_n}-a_n -a\in{[r,r+1]}, \bar{T}_k\leq \bar{R}_n(k), k\leq k_n,\bar{T}_j>\bar{F}_n(j)\right)\\&
		\leq \mathbb{P}\left( \bar{T}_k\leq \bar{R}_n(k), k\leq j,\bar{T}_j>\bar{F}_n(j)\right)\times\\&	\sup_{x\in{[F_n(j),R_n(j)]}}\mathbb{P}_x\left(\bar{T}_{k_n-j}-a_n -a\in{[r,r+1]}, \bar{T}_k\leq \bar{R}_n(k+j), k\leq k_n-j\right).
	\end{align}
 To bound the probability \eqref{eq41}, we apply the Markov property at time $l=[\frac{j}{3}]$,	
		\begin{align*}
		&\mathbb{P}(\bar{T}_{j}>\bar{F}_n(j),\bar{T}_{k}\leq \bar{R}_n(k),k\leq j )\\&\leq \mathbb{P}(\bar{T}_{k}\leq \bar{R}_n(k),k\leq l)\sup_{z\leq \bar{R}_n(l)}\mathbb{P}_z\left(\bar{T}_{j-l}>\bar{F}_n(j),\bar{T}_{k}\leq `\bar{R}_n(k+l), k\leq j-l \right).
		\end{align*}
		Set $\widehat{T_k}=\bar{T}_{j-l}-\bar{T}_{j-l-k}$, which is a random walk with the same law as $\bar{T_k}$. Then we obtain 
		\begin{align}
		&\nonumber\mathbb{P}_z\left(\bar{T}_{j-l}>\bar{F}_n(j),\bar{T}_{k}\leq \bar{R}_n(k+l), k\leq j-l \right)\\&\label{850}\leq \mathbb{P}_z\left(
		\bar{F}_n(j)<\bar{T}_{j-l}\leq \bar{R}_n(j),\bar{T_k}\geq F_n(j)-R_n(j-k),k\leq j-l\right).
		\end{align}
	   We 
		 bound the probability in $\eqref{850}$. We use a lower bound for the expression $(F_n(j)-R_n(j-k),k\leq j-l)$. Observe that  the function $ x\mapsto \frac{\log(x)}{x}$ is decreasing for $x\geq e$, and $$(k_n-j)b_n+1+kb_n\leq 2((k_n-j)b_n+1)kb_n,$$  then we have
	\begin{align*} 
	&F_n(j)-R_n(j-k)=\frac{-3}{2\theta^{*}}\left(\frac{j}{k_n}\log(k_nb_n)-\log(k_nb_n)+\log((k_n-j+k)b_n+1)\right)+\frac{\ln(b_n)}{\theta^*}-2c_n\\&\geq\frac{-3}{2\theta^{*}}\left(\log((j\lor e)b_n)-\log(k_nb_n)+\log((k_n-j)b_n+1)+\log(kb_n)+\log(2)\right)-2c_n\\&\geq \frac{-3}{2\theta^{*}}\left(\log((j\lor e)\land((k_n-j)+1))+\log(kb_n)+\log(2)\right)-2c_n , \forall k,j=1....k_n.
	\end{align*}		
	Applying again the Markov property at time $l$
	we get 
	\begin{align*}
	&\mathbb{P}_z\left(
	\bar{F}_n(j)<\bar{T}_{j-l}\leq \bar{R}_n(j),\bar{T_k}\geq F_n(j)-R_n(j-k),k\leq j-l \right)\\&\leq \mathbb{P}\left(\bar{T}_k\geq \frac{-3}{2\theta^{*}}\left(\log(kb_n)+\log(j\land((k_n-j)+1))+\log(2)\right)-2c_n ,k\leq l\right)\times& \\& \mathbb{P}_x\left(
	\bar{F}_n(j)-z<\bar{T}_{j-2l}\leq \bar{R}_n(j)-z\right).
	\end{align*}
		Assume that $\mathbf{(H_1)}$ or $\mathbf{(H_2)}$ hold,
		then by Lemmas \ref{MA}, \ref{Stone} or \ref{MAbis} for $n$ large enough  we obtain
		\begin{align*}
		&\mathbb{P}\left(\bar{T}_k\leq  \bar{R}_n(k), k\leq j,\bar{T}_j>\bar{F}_n(j)\right)\\&\underbrace{\leq \mathbb{P}\left(B_n(k)\leq\frac{1}{\sqrt{b_n}}\bar{R}_n(k),k\leq l \right)}_{\leq \frac{C}{\sqrt{j}}}	\times& \\&\underbrace{\mathbb{P}\left(B_n(k)\geq \frac{-3}{2\theta^{*}}\left(\log(j\land((k_n-j)+1))+\log(kb_n)+\log(2)\right)-2\frac{c_n}{\sqrt{b_n}} ,k\leq l\right)}_{\leq C\frac{1+\log(j\land((k_n-j)+1))}{\sqrt{j}}}\times&\\&\underbrace{\mathbb{P}_x\left(
		\frac{1}{\sqrt{b_n}}(\bar{F}_n(j)-z)<B_n(j-2l)\leq	\frac{1}{\sqrt{b_n}}( \bar{R}_n(j)-z)\right)}_{\leq C\frac{\log(j\land((k_n-j)+1))+1}{\sqrt{jb_n}}}\leq C\frac{ (1+(\log(j\land((k_n-j)+1)))^2}{\sqrt{b_n}j^{\frac{3}{2}}}.
		\end{align*}
		Using Lemma \ref{10} or \ref{7bis}, for $n$ large enough we have
		\begin{align*}
		 &\sup_{x\in{[F_n(j),R_n(j)]}}\mathbb{P}_x\left(\bar{T}_{k_n-j}-a_n-a\in{[r,r+1]}, \bar{T}_k\leq \bar{R}_n(k+j), k\leq k_n-j\right)\\&\leq C \frac{(1+2d_n)}{\sqrt{b_n}}\frac{ (1+\log(j\land((k_n-j)+1)))^2}{(k_n-j)^{\frac{3}{2}}},
		\end{align*}
		therefore, we conclude that 
		\begin{align*}
		&\mathbb{E}(\mathcal{E}_{n,R_n}(A)-\mathcal{E}_{n,F_n}(A))\leq C \frac{(1+2d_n)}{\sqrt{b_n}} k_n^{\frac{3}{2}}
		\sum_{j\leq k_n}\frac{(1+(\log(j\land((k_n-j)+1))))^4}{((k_n-j)+1)^{\frac{3}{2}}j^{\frac{3}{2}}}\\&\leq 2C\frac{(1+2d_n)}{\sqrt{b_n}}
		\sum_{j\leq [k_n/2]}\frac{(1+\log(j))^4}{j^{\frac{3}{2}}} \to_{n\rightarrow \infty} 0,
		\end{align*}
		as 	$\sum_{j\leq [k_n/2]}\frac{(1+\log(j))^4}{j^{\frac{3}{2}}}<\infty$.
	\end{proof}
	This lemma implies that the two extremal processes $\mathcal{E}_{n,F_n}$ and $\mathcal{E}_{n,R_n}$ have the same weak limit.  Consequently, the the same as  the o ne of $\mathcal{E}^{b_n}_{n}$ .	
	For $u\in{\mathcal{T}^{(n)}}$, we introduce
	$$H_{n}(u)=\{S_{u_k}\leq F_n(k), k\leq k_n \}$$  the set of individuals satisfying the $F_n$-barriers. The last step of the proof of our result is to show that, with high probability, the set of pairs of extremal particles that branch off at time $k\geq 1$ and stay all the time below the barrier $k\mapsto F_n(k)$ is vanishing in the large $n$-limit. This show that all particles contributing in the extremal process split from the root. 
	\begin{lemma}
		\label{11}
		With the same notation used before, we have, $$ \lim_{n\to \infty}\mathbb{E}\left(\#\{(u,v), |u\land v|\geq 1,H_n(u),H_n(v),S_u-m_n \in{A},S_v-m_n \in{A}\}\right)=0.$$  
	\end{lemma}
	\begin{proof}
		
		By considering the positions of any pairs of individuals $(u,v)$ at the generation $k_n$ and at their common ancestors $u\land v$ we have

	\begin{align*}
	&\mathbb{E}\left(\#\{(u,v), |u\land v|\geq 1, S_u- m_n \in{A}, S_v-m_n\in{A},H_n(u),H_n(v)\}\right){}\\& =\mathbb{E}\left(\sum_{j=1}^{k_n-1}\sum_{|w|=j}\mathbf{1}_{\{S_{w_i}\leq F_n(i), i\leq j \}} \sum_{(u_{j+1},v_{j+1})}\sum_{(u,v)}\mathbf{1}_{\{S_u-m_n \in{A}, S_v-m_n\in{A},S_{u_k}\leq F_n(k),S_{v_k}\leq F_n(k),j+1 \leq k \leq k_n\}}\right)
	\end{align*}

		where the double sum $\sum_{(u_{j+1},v_{j+1})}$ is over pairs $(u_{j+1},v_{j+1})$ of distinct children of $w=u\land v$ and  $\sum_{(u,v)}$ is over pairs $(u,v)$ such that $|u|=|v|=k_n$ and $u$ is a descendant of $u_{j+1}$, and $v$ is a descendant of $v_{j+1}$ .  
		Applying the Markov property at time $j+1$ we get   
		\begin{align}
	\nonumber&\mathbb{E}\left(\#\{(u,v), |u\land v|\geq 1, S^{b_n}_u- m_n \in{A}, S^{b_n}_v-m_n\in{A},H_n(u),H_n(v)\}\right)\\&\label{q}\leq \mathbb{E}\left(\sum_{j=1}^{k_n-1}\sum_{|w|=j}\mathbf{1}_{\{S_{w_i} \leq F_n(i) , i\leq j \}}\sum_{(u_{j+1},v_{j+1})}\mathbf{1}_{\{S_{u_{j+1}}\leq F_n(j+1), S_{v_{j+1}}\leq F_n(j+1)\}}\phi_{j,n}(S_{u_{j+1}})\phi_{j,n}(S_{v_{j+1}})\right),
		\end{align}
		where $$ \phi_{j,n}(z)=\mathbb{E}\left(\sum_{|u|=k_n-j-1}\mathbf{1}_{\{z+S_{u}-m_n \in{A},S_{u_{k}}+z\leq F_n(j+k+1),k\leq k_n-j-1\}}\right). $$
		Now using  Proposition \ref{M} we obtain,  
		\begin{align*}
		&\phi_{j,n}(z)= \mathbb{E}\left(e^{-\theta^{*} \bar{T}_{k_n-j-1}} \mathbf{1}_{\{z+\bar{T}_{k_n-j-1}- m_n+(k_n-j-1)b_nv \in{A}, \bar{T}_{k}+z\leq F_n(j+k+1)-kb_nv,k\leq k_n-j-1\}}\right).
		\end{align*}
		Summing with respect to the value of $\bar{T}_{k_n-j-1}-m_n+(k_n-j-1)b_nv$ we have 
\begin{align}
    \nonumber&\phi_{j,n}(z) \leq \frac{C   
	n^{\frac{3}{2}}}{b_n}e^{\theta^{*}(z-(j+1)b_nv)}\sum_{h\geq 0} e^{-\theta^*h}
 \\&\label{1}\times\mathbb{P}_{z-(j+1)b_nv}\left(\bar{T}_{k_n-j-1}-a_n-a \in{[h,h+1]}, \bar{T}_{k}\leq F_n(j+k+1)-kb_nv,k\leq k_n-j-1\right).
\end{align}
	Note that If $\mathbf{(H_2)}$ holds by Theorem \ref{Komlos}, we bound the quantity \eqref{1} by
\begin{align*}		
&\mathbb{P}_z\left(\bar{T}_{k_n-j-1}- m_n+( k_n-j-1)b_nv -a\in{[h,h+1]}, \bar{T}_{k}\leq F_n(j+k+1)-kb_nv,k\leq k_n-j-1\right)\\&\leq \mathbb{P}_{z-(j+1)b_n}\left(\mathcal{D}^{(k)}_n\right),
\end{align*}
 where
 \begin{align*}
    & \mathcal{D}^{(k)}_n={\left\{\sqrt{b_n}\hat{S}_{k_n-j-1}-a_{k_n-j-1}-a \in[h-d_n,h+1+d_n],\sqrt{b_n}\hat{S}_{k}\leq \bar{F}_n(k)+d_n, ,k\leq k_n-j-1\right\}}.
 \end{align*} 
 Now assume that either $\mathbf{(H_1)}$ or $\mathbf{(H_2)}$ hold. 
		Thanks to Lemma \ref{10} or \ref{7bis} we have,
		\begin{align*}
		&\phi_{j,n}(z)\leq C(1+2d_n) k_n^{\frac{3}{2}} e^{\theta^{*}(z-(j+1)b_nv)}\frac{(1-\frac{z-(j+1)b_nv}{\sqrt{b_n}})^{2}}{(k_n-j)^{\frac{3}{2}}}.
		\end{align*}
		 By replacing this in the equation \eqref{q}, we get
		\begin{align*}
		&\mathbb{E}\left(\#((u,v), |u\land v|\geq 1, S^{b_n}_u- m_n \in{A}, S^{b_n}_v-m_n\in{A},H_n(u),H_n(v))\right)\\& \leq C (1+2d_n) k_n^{\frac{3}{2}}  
		\mathbb{E}(\sum_{j=1}^{k_n-1}\frac{1}{(k_n-j)^{\frac{3}{2}}}\sum_{|w|=j}\mathbf{1}_{\{S_{w_i}\leq F_n(i),i\leq j\}}
		\times& \\
		&\sum_{(u_{j+1},v_{j+1})}\mathbf{1}_{\{S_{u_{j+1}}\leq F_n(j+1), S_{v_{j+1}}\leq F_n(j+1)\}} e^{\theta^{*}{S_{u_{j+1}}-(j+1)b_nv 
+S_{v_{j+1}}-(j+1)b_nv}}f_{n,j}(S_{u_{j+1}})f_{n,j}(S_{v_{j+1}}))
		\end{align*}
		where
		$f_{n,j}(u) =  
		(1+\frac{z-(j+1)b_nv}{\sqrt{b_n}})^{2}$.
	In the other hand we can bound the double sum  $$\sum_{(u_{j+1},v_{j+1})}\mathbf{1}_{\{S_{u_{j+1}}\leq F_n(j+1), S_{v_{j+1}}\leq F_n(j+1)\}} e^{\theta^{*}[S_{u_{j+1}}-(j+1)b_nv) 
		+S_{v_{j+1}}-(j+1)b_nv]}f_{n,j}(S_{u_{j+1}})f_{n,j}(S_{v_{j+1}})$$ by
		\begin{align*}
		    &(1+\frac{S_w-jb_nv}{\sqrt{b_n}})^4 e^{2\theta^{*}(S_{w}-jb_nv)}\\&\times  \mathbb{E}\left(\displaystyle\sum_{\substack{|u|=|v|=1\\ u\ne v}} (1+\frac{X^{(n)}_u-b_nv}{\sqrt{b_n}})^2 (1+\frac{X^{(n)}_v-b_nv}{\sqrt{b_n}})^2e^{\theta^{*}(X^{(n)}_{u}-b_nv+X^{(n)}_{v}-b_nv)} \right).
		\end{align*}

		Using independence between $X^{(n)}_u$ and $X^{(n)}_v$ for $u\ne v$ and the fact that $$\mathbb{E}\left(\sum_{|u|=1} (X_u^{(n)}-b_nv)e^{\theta^{*}(X^{(n)}_{u}-b_nv)}\right)=0,$$  we have 
	\begin{align*}
	&\mathbb{E}\left(\displaystyle\sum_{\substack{|u|=|v|=1\\ u\ne v}} (1+\frac{X^{(n)}_u-b_nv}{\sqrt{b_n}})^2 (1+\frac{X^{(n)}_v-b_nv}{\sqrt{b_n}})^2e^{\theta^{*}(X^{(n)}_{u}-b_nv+X^{(n)}_{v}-b_nv)} \right)\\&\leq\mathbb{E}\left(\displaystyle\sum_{\substack{|u|=|v|=1\\ u\ne v}}(1+(\frac{X^{(n)}_u-b_nv}{\sqrt{b_n}})^2)e^{\theta^{*}(X^{(n)}_{u}-b_nv)}(1+(\frac{X^{(n)}_v-b_nv}{\sqrt{b_n}})^2)e^{\theta^{*}(X^{(n)}_{v}-b_nv)} \right),
	\end{align*}
	then conditioning on $Z_{b_n}$ and using the following properties $$ \mathbb{E}\left((X_u^{(n)}-b_nv)^2e^{\theta^{*}(X^{(n)}_{u}-b_nv)}\right)=
	b_n\Lambda^{''}(\theta^{*})=
	b_n\sigma^2 \text{ and } \mathbb{E}\left(e^{\theta^{*}(X^{(n)}_{u}-b_nv)}\right)=m^{-b_n},$$ we obtain
	\begin{align*}
	&\mathbb{E}\left(\displaystyle\sum_{\substack{|u|=|v|=1\\ u\ne v}}(1+(\frac{X^{(n)}_u-b_nv}{\sqrt{b_n}})^2)e^{\theta^{*}(X_{u}-b_nv)}(1+(\frac{X^{(n)}_v-b_nv}{\sqrt{b_n}})^2)e^{\theta^{*}(X^{(n)}_{v}-b_nv)}|Z_{b_n} \right)\\&\leq (1+\sigma^2)^2 Z_{b_n}(Z_{b_n}-1)m^{-2b_n},
	\end{align*} 
	as a consequence, there exists a constant $C>0$ such that 
	\begin{align*}
		&\mathbb{E}\left(\displaystyle\sum_{\substack{|u|=|v|=1\\ u\ne v}}(1+(\frac{X^{(n)}_u-b_nv}{\sqrt{b_n}})^2)e^{\theta^{*}(X_{u}-b_nv)}(1+(\frac{X^{(n)}_v-b_nv}{\sqrt{b_n}})^2)e^{\theta^{*}(X_{v}-b_nv)} \right)\leq C,
	\end{align*}
			which leads to the following inequality 
		\begin{align*}
		&\mathbb{E}\left(\#((u,v), |u\land v|\geq 1, S^{b_n}_u- m_n \in{A}, S^{b_n}_v-m_n\in{A},H_n(u),H_n(v))\right)\\& \leq C (1+2d_n)k_n^{\frac{3}{2}} \mathbb{E}\left(\sum_{j=1} 
		^{k_n-1}\frac{1}{(k_n-j)^{\frac{3}{2}}}\sum_{|w|=j}(1+\frac{S_w-jb_nv}{\sqrt{b_n}})^4e^{2\theta^{*}(S_{w}-jb_nv)}\mathbf{1}_{\{S_{w_i}\leq F_n(i),i\leq j\}}\right).
		\end{align*}
		On the other hand by Proposition \ref{M} we obtain
		\begin{align*}
		&\mathbb{E}\left(\#((u,v), |u\land v|\geq 1, S_u- m_n \in{A}, S_v-m_n\in{A},H_n(u),H_n(v))\right)\\
		&\leq C(1+2d_n) k_n^{\frac{3}{2}} \mathbb{E}\left(\sum_{j=1} 
		^{k_n-1}\frac{1}{(k_n-j)^{\frac{3}{2}}}(1+\frac{\bar{T_j}}{\sqrt{b_n}})^4e^{\theta^{*}\bar{T}_j}\mathbf{1}_{\{\bar{T}_i\leq \bar{F}_n(i), i\leq j \}}\right).
		\end{align*}
		Summing with respect to the values of $\bar{T_j}-\bar{F}_n(j)$, by Lemma \ref{10} or \ref{7bis} and for $n$ large enough  we get
\begin{align*}		
		&\mathbb{E}\left(\#((u,v), |u\land v|\geq 1, S_u- m_n \in{A}, S_v-m_n\in{A},H_n(u),H_n(v))\right)\\&\leq C(1+2d_n) k_n^{\frac{3}{2}}\sum_{j=1} 
		^{k_n-1}\frac{1}{(k_n-j)^{\frac{3}{2}}}\sum_{r=0}^{\infty}\mathbb{E}\left(
		e^{\theta^{*}\bar{T}_j}(1+\frac{\bar{T_j}}{\sqrt{b_n}})^{4}\mathbf{1}_{\{\bar{T}_i\leq \bar{F}_n(i),\bar{T}_j-\bar{F}_n(j) \in{[-r-1,-r]}, i\leq j\}}\right) 
	\\&\leq C(1+2d_n)e^{-\theta^{*}c_n}k_n^{\frac{3}{2}}\sum_{j=1} 
	^{k_n-1}\frac{1}{(k_n-j)^{\frac{3}{2}}}\sum_{r=0}^{\infty}(1+r)^{4}e^{-\theta^{*}r}\\&\times\mathbb{P}\left(\bar{T}_i\leq \bar{F}_n(i),\bar{T}_j-\bar{F}_n(j) 
		\in{[-r-1,-r]},i\leq j\right)\\&\leq \frac{C(1+2d_n)}{\sqrt{b_n}}e^{-\theta^{*}c_n}\sum_{j=1} 
	^{\lfloor\frac{k_n}{2}\rfloor }\frac{k_n^{\frac{3}2{}}}{(k_n-j)^{\frac{3}{2}}j^{\frac{3}{2}}} \to_{n\to\infty}0,
		\end{align*}     
		where we used that $\sum_{j=1}^{\lfloor\frac{k_n}{2}\rfloor}\frac{k_n^{\frac{3}2{}}}{(k_n-j)^{\frac{3}{2}}j^{\frac{3}{2}}} \leq 2^{\frac{3}{2}} \sum_{j=1} ^{\lfloor\frac{k_n}{2}\rfloor} \frac{1}{i^{\frac{3}{2}}}<\infty.$
	\end{proof}
	Now we are ready to prove our main result. 

	\begin{proof}[Proof of Theorem \ref{Theorem}]
	Let $\phi\in {\mathcal{C}^{l,+}_{b}}$, with support $A=[a,\infty)$ where $a\in{\mathbb{R}}$. We have to show that 
	\begin{align}
	\label{T1}
	    &\lim_{n\to \infty}\mathbb{E}\left(e^{-\sum_{u\in{\mathcal{H}_{k_n}}} \phi(S_{u}^{(n)}-m_n)}\right)= \mathbb{E}\left(\exp\left(-Z_{\infty}\frac{1}{\sqrt{2\pi\sigma^2}}\int_{}e^{-\theta^{*}y}(1-e^{-\phi(y)})dy\right)\right).
	\end{align}

	First introduce  $$G_n=\{ {\nexists} (u,v),|u\land v|\geq 1,S_u-m_n\geq a,S_v-m_n\geq a\}.$$
	By Lemma \ref{lemma 14} and Proposition \ref{pro11}, it is enough to prove \eqref{T1} for the extremal process  $$\mathcal{E}_{n,F_n}=\sum_{u\in{\mathcal{H}_{k_n}}}\delta_{S_u-m_n}\mathbf{1}_{H_n}.$$
	where $H_n={\left\{S_{u_{k}}\leq F_n(k), \forall  k\leq k_n\right\}}.$
	Using Lemma \ref{11}, we have $\mathbb{P}(G_n^c)\to_{n\rightarrow \infty}0$, therefore
	
		\begin{align*} 
		&\mathbb{E}\left(\exp-\left(\sum_{u\in{\mathcal{H}_{k_n}}} \phi(S_{u}-m_n)\mathbf{1}_{H_n}\right)\right) \\&=\mathbb{E}\left(\mathbf{1}_{G_n}\exp-\left(\sum_{u\in{\mathcal{H}_{k_n}}} \phi(S_{u_{k_n}}-m_n)\mathbf{1}_{H_n}\right)\right)+o(1)\\&=\mathbb{E}\left(\mathbf{1}_{G_n}\prod_{|w|=1}\left(\exp-\left(\sum \limits_{\underset{|u|=k_n}{u>w}} \phi(S_{u_{k_n}}-m_n)\mathbf{1}_{H_n}\right)\right)\right)+o(1).
	\end{align*}
	 Observe that $\exp{(-\sum_{i=1}^{n}x_i)}=1+\sum_{i=1}^{n}(e^{-x_i}-1) \text{ if  there exists at most $i$ such that }$ $x_i \ne 0$. Hence  using that under $G_n$, for all $w$ at the first generation, at most one descendant reaches level $m_n$, we get, 
	\begin{align*}  	
	&\mathbb{E}\left(\exp-\left(\sum_{u\in{\mathcal{H}_{k_n}}} \phi(S_{u}-m_n)\mathbf{1}_{H_n}\right)\right)\\&=\mathbb{E}\left(\mathbf{1}_{G_n}\prod_{|w|=1}\left(1+\sum \limits_{\underset{|u|=k_n}{u>w}}\left(\exp(-\phi(S_{u}-m_n+S_w))-1\right)\mathbf{1}_{H_n}\right)\right)+o(1)\\&=\mathbb{E}\left(\prod_{|w|=1}\left(1+\sum \limits_{\underset{|u|=k_n}{u>w}}\left(\exp(-\phi(S_{u}-m_n+S_w))-1\right)\mathbf{1}_{H_n}\right)\right)+o(1),
	\end{align*}
	 using again that $G_n^c$ is an event of asymptotically small probability, and that this product of random variable is bounded by $1$. We now apply the Markov property  at time one to obtain 
	\begin{align}
	\label{100}
	&\mathbb{E}\left(\exp-\left(\sum_{u\in{\mathcal{H}_{k_n}}} \phi(S_{u}-m_n)\mathbf{1}_{H_n}\right)\right) =\mathbb{E}\left(\prod_{|w|=1} 1+\psi_n(S_w)\mathbf{1}_{\left\{S_w-b_nv\leq \bar{F}_n(1) \right\}}\right)+o(1),
	\end{align}
	where 
	\begin{align} 
	&\psi_n:x\mapsto \mathbb{E}\left(\sum_{|u|=k_n-1}\left(e^{-\phi(S_{u}-m_n+x)}-1\right)\mathbf{1}_{\{S_{u_{k}}\leq F_n(k+1)-x, k< k_n\}}\right),
	\end{align}
	using Proposition \ref{M} we have,
	\begin{align*}
	&\psi_n(x)=\mathbb{E}\left(e^{-\theta^{*}\bar{T}_{k_n-1}}e^{-\phi(\bar{T}_{k_{n}-1}-\alpha_{b_n}+x)}-1)\mathbf{1}_{\{\bar{T}_{k}\leq \bar{F}_n(k+1)-(x-b_nv),k<k_n\}}\right)
	\end{align*} 
	where $\alpha_{b_n}=b_nv-\frac{3}{2\theta^{*}}\ln(n)+\frac{\ln(b_n)}{\theta^{*}}$.

 Using Lemmas \ref{Stone barrier} or Corollary \ref{stonebis}, depending on whatever we work under $(\mathbf{H}_1)$ or $(\mathbf{H}_2)$, we  obtain the following approximation,
		\begin{align*}
&\mathbb{E}\left(e^{-\theta^{*}\bar{T}_{k_n-1}}e^{-\phi(\bar{T}_{k_{n}-1}-\alpha_{b_n}+x)}-1)\mathbf{1}_{\{\bar{T}_{k}\leq \bar{F}_n(k+1)-(x-b_nv),k<k_n\}}\right)	\\&\sim_{n \to \infty}\frac{1}{\sqrt{2\pi\sigma^2}}L(\frac{-(x-b_nv)}{\sqrt{b_n}})e^{\theta^{*}(x-b_nv)}\int e^{-\theta^{*}y}(e^{-\phi(y)}-1)dy.
      \end{align*}
Plugging this in Equation  \eqref{100} we get
	\begin{align*}
&\mathbb{E}\left(\exp-\left(\sum_{u\in{\mathcal{H}_{k_n}}} \phi(S_{u}-m_n)\mathbf{1}_{H_n}\right)\right)\\&\sim_{n \to \infty}\mathbb{E}\left(\prod_{|w|=1}\mathbb{E}(1+L(-\frac{S_w-b_nv}{\sqrt{b_n}})e^{\theta^{*}(S_w-b_nv)}\mathbf{1}_{\left\{S_w-b_nv\leq\bar{F}_n(1)\right\}}\frac{1}{\sqrt{2\pi\sigma^2}}\int e^{-\theta^{*}y}(e^{-\phi(y)}-1)dy) \right).
	\end{align*}
	Recall that $v=\frac{\kappa(\theta^{*})}{\theta^{*}}$, then by Girsanov transform  we have $$\mathbb{E}\left(L(\frac{-(S_w-b_nv)}{\sqrt{b_n}})e^{\theta^{*}(S_w-b_nv)}\mathbf{1}_{\left\{S_w-b_nv\leq\bar{F}_n(1)\right\}}\right)=m^{-b_n}\mathbb{E}\left(L(\frac{S_1^{(n)}}{\sqrt{b_n}})\mathbf{1}_{\left\{\frac{S_1^{(n)}}{\sqrt{b_n}}\leq \frac{\bar{F}_n(1)}{\sqrt{b_n}}\right\}}\right)$$
	where $S_1^{(n)}$ is a centred random walk with variance $b_n$. Fix $A>0$
	\begin{align*}
	    	&\mathbb{E}\left(L(\frac{-S_1^{(n)}}{\sqrt{b_n}})\mathbf{1}_{\left\{\frac{S_1^{(n)}}{\sqrt{b_n}}\leq\bar{F}_n(1)\right\}}\right)\\&= 	\mathbb{E}\left(L(\frac{-S_1^{(n)}}{\sqrt{b_n}})\mathbf{1}_{\left\{-A\leq \frac{S_1^{(n)}}{\sqrt{b_n}}\leq \frac{\bar{F}_n(1)}{\sqrt{b_n}}\right\}}\right)+	\mathbb{E}\left(L(\frac{-S_1^{(n)}}{\sqrt{b_n}})\mathbf{1}_{\left\{\frac{S_1^{(n)}}{\sqrt{b_n}}\leq -A \right\}}\right).
	\end{align*} 
	As the  function $x \mapsto L(x)\mathbf{1}_{\left\{x \in [-A, \frac{\bar{F}_n(1)}{b_n}]\right\}}$ is bounded, by central limit theorem we have \begin{align}
	\label{mv}
	   \lim_{n\to \infty} \mathbb{E}\left(L(\frac{-S_1^{(n)}}{\sqrt{b_n}})\mathbf{1}_{\left\{-A\leq \frac{-S_1^{(n)}}{\sqrt{b_n}}\leq \frac{\bar{F}_n(1)}{\sqrt{b_n}}\right\}}\right)= \frac{1}{\sqrt{2\pi}}\int_{-A}^{0}L(-y)e^{-\frac{y^2}{2}} dy.
	\end{align} 
Additionally, we have by Equation \eqref{renewl} and the Markov inequality
\begin{align*}
  \left|\mathbb{E}\left(L(\frac{-S_1^{(n)}}{\sqrt{b_n}})\mathbf{1}_{\left\{\frac{S_1^{(n)}}{\sqrt{b_n}}\leq -A \right\}}\right)\right| &\leq C  \left|\mathbb{E}\left((1+\frac{-S_1^{(n)}}{\sqrt{b_n}})\mathbf{1}_{\left\{\frac{S_1^{(n)}}{\sqrt{b_n}}\leq -A \right\}}\right)\right|\\& \leq C\frac{\mathbb{E}(|S^{(n)}_1|)}{\sqrt{b_n}A}+\mathbb{E}\left( \frac{|S_1^{(n)}|}{\sqrt{b_n}A} \frac{|S_1^{(n)}|}{\sqrt{b_n}}\mathbf{1}_{\left\{\frac{S_1^{(n)}}{\sqrt{b_n}}\leq -A \right\}}\right) 
  \leq \frac{C}{A}.
\end{align*}

Therefore, we have
\[
  \limsup_{n \to \infty} \left|  \mathbb{E}\left(L(\frac{-S_1^{(n)}}{\sqrt{b_n}})\mathbf{1}_{\left\{\frac{S_1^{(n)}}{\sqrt{b_n}}\leq \frac{\bar{F}_n(1)}{\sqrt{b_n}}\right\}}\right) - \frac{1}{\sqrt{2\pi}}\int_{-A}^{0}L(-y)e^{-\frac{y^2}{2}} dy\right|\leq \frac{C}{A}.
\]
Thus, letting $A \to \infty$ in \eqref{mv}, by \eqref{mhhh} we obtain that 
\begin{align*}
     &\lim_{n \to \infty} \mathbb{E}\left(L(\frac{-S_1^{(n)}}{\sqrt{b_n}})\mathbf{1}_{\left\{\frac{S_1^{(n)}}{\sqrt{b_n}}\leq \frac{\bar{F}_n(1)}{\sqrt{b_n}}\right\}}\right) = \frac{1}{\sqrt{2\pi}}\int_{-\infty}^{0}L(-y)e^{-\frac{y^2}{2}} dy\\&=\mathbb{E}\left(L(\frac{S_1^{(n)}}{\sqrt{b_n}})\mathbf{1}_{\left\{\frac{S_1^{(n)}}{\sqrt{b_n}}\geq 0\right\}}\right)=L(0)=1.
\end{align*}
	As a consequence we obtain
	\begin{align*}
&\mathbb{E}\left(\prod_{|w|=1}\mathbb{E}\left(1+L(-\frac{S_w-b_nv}{\sqrt{b_n}})e^{\theta^{*}(S_w-b_nv)}\mathbf{1}_{\left\{S_w-b_nv\leq\frac{\bar{F}_n(1)}{b_n}\right\}}\frac{1}{\sqrt{2\pi\sigma^2}} \int e^{-\theta^{*}y}(e^{-\phi(y)}-1)dy\right) \right)\\&\sim_{n\to \infty}\mathbb{E}\left(\left(1-\frac{1}{\sqrt{2\pi\sigma^2}}m^{-b_n}\int e^{-\theta^{*}y}(1-e^{-\phi(y)})dy\right)^{\#\left\{u,|u|=1\right\}} \right)\\&=\mathbb{E}\left(\left(1-\frac{1}{\sqrt{2\pi\sigma^2}}m^{-b_n}\int e^{-\theta^{*}y}(1-e^{-\phi(y)})dy\right)^{Z_{b_n}} \right).
\end{align*}
	Finally,  applying dominated convergence theorem and by assumption \eqref{F} we deduce that \eqref{T1} holds  for all function $\phi\in{\mathcal{C}^{l,+}_{b}}$, which concludes the proof using Remark \ref{remark}.
	\end{proof}

	\paragraph{Acknowledgements.}
 I want to thank  Bastien Mallein for introducing me this subject and for the very helpful discussions.
 
\noindent\begin{minipage}{0.1\textwidth}
\end{minipage}\quad
\begin{minipage}{0.85\textwidth}
This program has received funding from the European Union's Horizon 2020 research and innovation program under the Marie Skłodowska-Curie grant agreement No 754362.
\end{minipage}

\bibliographystyle{plain}
\bibliography{R}
\end{document}